\documentclass{birkjour}
\usepackage{amsmath}
\usepackage{amssymb}
\usepackage{amsthm}
\usepackage{enumerate}
\usepackage[mathscr]{eucal}
\usepackage{xcolor}
\theoremstyle{plain}
\usepackage{tikz}
\usepackage{soul}
\usepackage[normalem]{ulem}
\newtheorem{theorem}{Theorem}[section]
\newtheorem{lemma}[theorem]{Lemma}
\newtheorem{prop}[theorem]{Proposition}
\theoremstyle{definition}
\newtheorem{definition}[theorem]{Definition}
\newtheorem{remark}[theorem]{Remark}

\newtheorem{example}[theorem]{Example}

\newtheorem{cor}[theorem]{Corollary}
\theoremstyle{remark}
\mathchardef\mhyphen="2D 

\begin{document}
	
	\title[On anti-coproximinal and strongly anti-coproximinal subspaces ]
	{On anti-coproximinal and strongly \\anti-coproximinal  subspaces of function spaces}

\author[Sohel]{Shamim Sohel}
\address{Department of mathematics\\ Jadavpur University\\Kolkata 700032\\West Bengal\\India}
\email{shamimsohel11@gmail.com}

\author[Ghosh]{Souvik Ghosh}
\address{Department of mathematics\\ Jadavpur University\\Kolkata 700032\\West Bengal\\India}
\email{sghosh0019@gmail.com}

\author[Sain]{Debmalya Sain}
\address{Department of mathematics\\ Indian Institute of Information Technology\\Raichur\\Karnataka 584135\\ India}
\email{saindebmalya@gmail.com}

\author[Paul]{Kallol Paul}

\address{Vice-Chancellor\\
	Kalyani University\\
	West Bengal 741235 \\and 
	Professor (on lien)\\ Department of mathematics\\ Jadavpur University\\Kolkata 700032\\West Bengal\\India}
\email{kalloldada@gmail.com}

	
	

	\maketitle
	
\begin{abstract}
	 The purpose of this article is to study the anti-coproximinal and strongly anti-coproximinal subspaces of the Banach space of all bounded (continuous) functions. We obtain a tractable necessary condition for a subspace to be stronsgly anti-coproximinal. We prove that for a subspace $\mathbb{Y}$ of a Banach space $\mathbb{X}$ to be strongly anti-coproximinal, $\mathbb Y$ must contain all w-ALUR points of $\mathbb{X}$ and intersect every maximal face of $B_{\mathbb{X}}.$ We also observe that the subspace $\mathbb{K}(\mathbb{X}, \mathbb{Y})$ of all compact operators between the Banach spaces $ \mathbb X $ and $ \mathbb Y$ is strongly anti-coproximinal in the space $\mathbb{L}(\mathbb{X}, \mathbb{Y})$ of all bounded linear operators between $ \mathbb X $ and $ \mathbb Y$, whenever $\mathbb{K}(\mathbb{X}, \mathbb{Y})$ is a proper subset of $\mathbb{L}(\mathbb{X}, \mathbb{Y}),$ and the unit ball $B_{\mathbb{X}}$ is the closed convex hull of its strongly exposed points. 
\end{abstract}

	\section{Introduction}
		
	The concept of best coapproximation, introduced by Franchetti and Furi  \cite{FF}, plays an important role in understanding the approximation properties of the subspaces of a Banach space. Consequently, the best coapproximation problem in Banach spaces has been explored from several perspectives \cite{KL, LT, N, RS, Rao1, Rao2}. We should perhaps begin with the fact even in the finite-dimensional case, the existence of best coapproximation is not guaranteed, to illustrate the non-triviality of the problem. Indeed, as observed by  Papini et al. in \cite{PS}, the best coapproximation problem is equivalent to the existence of a norm-one projection onto the concerned subspace. For some of the nice continuations of this connection, we refer the readers to \cite{Bruck, Bruck1, LT, M2, Rao, S}. Let us also mention here that the notion of Birkhoff-James orthogonality plays a crucial role in understanding  the best coapproximation problem (see \cite{FF, SSGP, SSGP2, SSGP3}). Motivated by this fact, two special types of subspaces of a Banach space, christened anti-coproximinal and strongly anti-coproximinal subspaces, have been introduced recently in \cite{SSGP3}.  In this article, we explore the (strongly) anti-coproximinal subspaces of the Banach space  of all bounded (continuous) functions, including the space of all bounded linear operators between Banach spaces. Our study relates the geometric properties of the underlying Banach spaces with the approximation properties of subspaces of function spaces, to illustrate the specialty of the (strongly) anti-coproximinal subspaces. 

	\medskip
	Let $\mathbb{X}, \mathbb{Y}$ denote Banach spaces over the field $\mathbb{K},$ either real or complex. For $\epsilon >0,$  let us set $\mathcal{D}(\epsilon)= \{z \in \mathbb{K}: |z| \leq \epsilon \}$  and let $\mathbb{T}$ denote the unit circle in the complex plane.  We use the notations $B_\mathbb{X}$ and $S_{\mathbb{X}}$ to denote the unit ball and unit sphere of $\mathbb{X},$ respectively.  Let $\mathbb{L}(\mathbb{X}, \mathbb{Y})$ $(\mathbb{K}(\mathbb{X}, \mathbb{Y}))$ be the space of all bounded (compact) linear operators between $\mathbb{X}$ and $\mathbb{Y}.$ The dual space of $\mathbb{X}$ is denoted by $\mathbb{X}^*.$ For a non-zero $x \in \mathbb{X},$  $x^* \in S_{\mathbb{X}^*}$ is said to be a supporting functional at $x$ if $x^*(x)=\|x\|.$ The set of all supporting functionals at $x$ is denoted by $J(x),$ i.e., $J(x)=\{x^* \in S_{\mathbb{X}^*}: x^*(x)=\|x\|\}.$ A non-zero  element $x \in \mathbb{X} $ is said to be smooth if $J(x)$ is a singleton. The collection of all smooth points in $\mathbb{{X}}$ is denoted by $ Sm(\mathbb{X}).$ A Banach space $\mathbb{X}$ is said to be smooth if $Sm(\mathbb{X})=\mathbb{X} \setminus \{0\}.$ For a subspace $ \mathbb{Y} $ of $ \mathbb{X}, $	we use the notation $\mathcal{J}_{\mathbb{Y}} = \{ y^* \in S_{\mathbb{X}^*} : y^*(y) = 1, \text{for some}~  y \in Sm(\mathbb{{X}})\cap S_{\mathbb{Y}} \}.$ It is easy to check that $\mathcal{J}_{\mathbb{Y}} \subseteq Ext(B_{\mathbb{X}^*}).$ Whenever $ Sm(\mathbb{X}) \cap S_\mathbb{{Y}}= \emptyset, $ we define $\mathcal{J}_\mathbb{{Y}}= \emptyset.$ 	The convex hull of a set $S$ is denoted as $co(S).$ For a convex set $C,$ an element $x \in C$ is said to be an extreme point if $x= (1-t) y+ tz,$ for some $t \in (0,1)$ and some $y,z \in C$ implies that $x=y=z.$ The set of all extreme points of $C$ is denoted by $Ext(C).$ A finite-dimensional real Banach space is said to be polyhedral if $B_{\mathbb{X}}$ is a polyhedron, or, equivalently, if $Ext(B_{\mathbb{X}})$ is finite. 
	A convex set $F \subset S_{\mathbb{X}}$ is said to be a face of $B_{\mathbb{X}} $ if for any $ y, z \in B_{\mathbb{X}},$ $\frac{1}{2}(y+z) \in F$ implies that $y, z \in F.$  $F$ is called a maximal face if for any face $F'$ of $B_\mathbb{X},$ $F \subset F'$ implies $F=F'.$  For any face $ F $ of $B_{\mathbb{X}}$ and any $x^* \in S_{\mathbb{X}^*},$ we say that $x^*$ supports $F$ if $x^*(x)=1~ \forall x \in F.$ We use the notation $int(F)$ to denote the relative interior of a face $F$ endowed with the subspace topology of $F.$
		The space  $ \mathbb X$ is said to be strictly convex if $Ext(B_{\mathbb{X}})= S_{\mathbb{X}}.$ An element $x \in S_\mathbb{X}$ is said to be rotund if for some  $y \in B_\mathbb{X},$ $\|\frac{x+y}{2}\|=2$ implies $x=y.$ In a strictly convex space every element of the unit sphere is  rotund. 
An element $x \in S_{\mathbb{X}}$ is said to be weakly almost locally uniformly rotund or w-ALUR point if given any $\{x_m^*\}_{m \in \mathbb{N}} \subset B_{\mathbb{X}^*}$ and any $\{x_n\}_{n \in \mathbb{N}} \subset B_\mathbb{X} $, $$\lim_{m\to \infty}\lim_{n\to\infty} x_m^* \bigg(\frac{x_n+x}{2}\bigg)=1$$ implies that $x_n \overset{w}{\longrightarrow} x$.
		 The space $\mathbb{X}$ is said to be  w-ALUR if each element of $S_\mathbb{X}$ is  w-ALUR. 	 An element $x \in S_{\mathbb{X}}$ is said to be an exposed point of $B_{\mathbb{X}}$ if there exists $x^* \in J(x)$ such that $x^*(y) < 1= x^*(x),$ for any $y \in S_{\mathbb{X}} \setminus \{x\}.$ Clearly, every exposed point of $B_{\mathbb{X}}$ is also an extreme point of $B_{\mathbb{X}}$.  We say $x \in S_{\mathbb{X}}$ to be a strongly exposed point of $B_{\mathbb{X}}$ if there exists $x^* \in J(x)$ such that for any sequence $\{x_n\} \subset  B_{\mathbb{X}},$ $x^*(x_n) \longrightarrow
	  1=x^*(x)$ implies that $x_n \longrightarrow x.$ The set of all  exposed points and strongly exposed points  of $B_{\mathbb{X}}$ are denoted by $Exp(B_{\mathbb{X}})$ and $s$-$Exp(B_{\mathbb{X}})$, respectively.  
	  
	\smallskip
		For any element $ x \in \mathbb{X}, $ and any subspace $\mathbb Y$ of $ \mathbb{X}, $  $ y_0 \in \mathbb{Y} $ is said to be a best coapproximation (see \cite{FF}) to $ x $ out of $ \mathbb Y$ if $ \| y_0 - y \| \leq \| x - y \| $ for all $ y \in \mathbb{Y}. $  As mentioned previously, we require the concept of Birkhoff-James orthogonality to gain a better understanding of the best coapproximation problem. Given $ x, y  \in \mathbb{X}$, we say that $ x $ is Birkhoff-James orthogonal \cite{B, J}  to $ y, $ written as $ x \perp_B y, $ if $ \| x+\lambda y \| \geq \| x \|, $ for all $ \lambda \in \mathbb{K}. $ It is clear that $ y_0 \in \mathbb{Y} $ is a best coapproximation to $ x $ out of $ \mathbb{Y} $ if and only if 
	$$ \mathbb{{Y}} \perp_B (x- y_0), \,\, i.e., \,\, y \perp_B (x-y_0) ~  \forall  y \in \mathbb{Y}. $$ 

   Given $\epsilon \in [0, 1)$ and $ x, y \in \mathbb{X},$  $x$ is said to be $\epsilon$-Birkhoff-James orthogonal \cite{C} to $y,$ written as $ x \perp_B^\epsilon y,$ if 
\[\|x + \lambda y\| \geq \|x\| -  \epsilon \|\lambda y\|~ \text{for ~every ~} ~\lambda \in \mathbb{K}. \]

The above definition, in conjunction with the previously mentioned relation between Birkhoff-James orthogonality and the best coapproximation, naturally leads us to the following definition of  $\epsilon$-best coapproximation, introduced in \cite{SSGP3}:\\

Let $ \epsilon \in [0,1).$  For a subspace $\mathbb Y$ and a  given $ x \in \mathbb{X},  $ we say that $ y_0 \in \mathbb{Y} $ is an  $\epsilon$-best coapproximation to $ x $ out of $ \mathbb{{Y}} $ if 
$$ \mathbb{{Y}} \perp_B^\epsilon (x- y_0), \,\, i.e., \,\, y \perp_B^\epsilon (x-y_0) \,  \forall  y \in \mathbb{Y}. $$

As noted in \cite{SSGP3},  the definitions of best coapproximation and $\epsilon$-best coapproximation motivate us to study the following two special types of subspaces of a Banach space:

\begin{definition}\label{epsilon} 
(i)  A subspace  $\mathbb{{Y}}$ of  $\mathbb{X}$ is said to be an \textit{anti-coproximinal subspace} of $\mathbb{X}$ if for any $x \in \mathbb{X} \setminus \mathbb{Y},$ there does not exist a best coapproximation to $x$ out of $\mathbb{Y}.$ Equivalently, a subspace $\mathbb{Y}$ is anti-coproximinal in $\mathbb{X}$ if  for any nonzero $x \in \mathbb{X},$ $\mathbb{Y} \not\perp_B x.$\\
		(ii) A subspace  $\mathbb{{Y}}$  of $\mathbb{X}$ is said to be a \textit{strongly anti-coproximinal subspace} of $\mathbb{X}$ if for any given $x \in \mathbb{X} \setminus \mathbb{Y}$ and for any $\epsilon \in [0, 1),$ there does not exist an $\epsilon$-best coapproximation to $x$ out of $\mathbb{Y}.$  Equivalently, a subspace  $\mathbb{Y}$ is strongly anti-coproximinal if  for any nonzero $x \in \mathbb{X}$ and for any $\epsilon \in [0, 1),$ $\mathbb{Y} \not\perp_B^{\epsilon} x.$ 

\end{definition}
 
For any nonempty set $K,$ $\ell_\infty(K, \mathbb{X})$ stands for the space of all bounded functions from $K$ to $\mathbb{X}.$	Given a compact Hausdorff topological space $K$ and a Banach space $\mathbb{X}$, we write $C(K, \mathbb{X} )$ to denote the space of  bounded continuous functions from $K$ to $\mathbb{X}$, endowed with the sup norm, i.e., 
$$C(K, \mathbb{X})=\{f \, | \, f: K \to \mathbb{X}\, \, \mbox{is continuous} \, \mbox{and} \, \sup_{k \in K} \|f(k)\| < \infty\}.$$
Given a locally compact Hausdorff space $K$ and a Banach space $\mathbb{X}$, the space $C_0(K, \mathbb{X})$ is the space of  bounded continuous functions $f$ having the property that for any $\epsilon >0,$ there exists a compact set $\Gamma \subset K$ such that $\|f(k)\|< \epsilon,$ for any $k \in K\setminus \Gamma.$ Whenever $K$ is compact, $C_0(K, \mathbb{X})= C(K, \mathbb{X}).$
	For a function $f \in C_0(K, \mathbb{X}),$ the norm attainment set of $f,$ denoted by $M_f,$ is defined as $	M_f= \{k\in K: \|f(k)\|= \|f\|\}.$ Whenever $\mathbb{X}= \mathbb{R}$ or $ \mathbb{C},$ we use the standard  notations  $C(K)$ and $C_0(K)$  in place of $C(K, \mathbb{X})$ and $C_0(K, \mathbb{X}),$ respectively. 
Note that for any  $f \in S_{C_0(K, \mathbb{X})},$ $ M_f$ is non-empty and compact.\\

The present article is divided into three sections, including the introductory one. In the preliminary section, we provide a characterization of approximate Birkhoff-James orthogonality in the function space as well as in the space of bounded linear operators. The main results of this article consist of three different parts. First we study the anti-coproximinal and the strongly anti-coproximinal subspaces of a Banach space, and obtain some useful  necessary and sufficient conditions for the same. In the next part, we consider the spaces $\ell_\infty(K)$ and $C_0(K)$ and find computationally effective characterizations of  these special types of subspaces. We show that the notions of anti-coproximinality and strong anti-coproximinality coincide in both $\ell_\infty(K)$ and $C_0(K)$. Moreover, we study these subspaces in sequence spaces,   $c_0, c \, \, \mbox{and}\,\,  \ell_\infty.$ In the last part, we consider the space of all bounded operators $\mathbb{L}(\mathbb{X}, \mathbb{Y})$ and investigate  when  $\mathbb{K}(\mathbb{X}, \mathbb{Y})$ is strongly anti-coproximinal in $\mathbb{L}(\mathbb{X}, \mathbb{Y}).$ We use the so-called \emph{ B\v S-Property} of operators to investigate anti-coproximinality of  $\mathbb{K}(\mathbb{X}, \mathbb{Y})$ in   $\mathbb{L}(\mathbb{X}, \mathbb{Y}).$  We  also prove that the set of norm attaining operators satisfying the \emph{ B\v S-Property} is dense in $\mathbb{L}(\mathbb{X}, \mathbb{Y}),$  if $\mathbb X $  satisfies the Radon-Nikodym Property. Finally, we provide  a sufficient condition for strong anti-coproximinality of a subspace  in $\mathbb{L}(\mathbb{H}),$ for any Hilbert space $ \mathbb{H}. $

\section{preliminaries}
Throughout this article, we require several results involving Birkhoff-James orthogonality and its approximate version. Therefore, mainly for the convenience of the readers, it is worth mentioning the relevant results in the context of the present study.

\begin{lemma}\label{James}\cite[Th. 2.1]{J}
	Let $\mathbb{X}$ be a Banach space and let $x, y \in \mathbb{X}.$  Then $x \perp_B y$ if and only if there exists $x^* \in J(x)$ such that $x^*(y)=0.$
\end{lemma}

\begin{lemma}\label{Chem}\cite[Th. 2.3]{CSW}
	Let $\mathbb{X}$ be a Banach space. Suppose $\epsilon \in [0,1)$ and let $x, y \in \mathbb{X}.$  Then $x \perp_B^\epsilon y$ if and only if  there exists $x^* \in J(x)$ such that $|x^*(y)| \leq \epsilon\|y\|.$
\end{lemma}



A more general version of Lemma \ref{James} has been obtained in \cite[Cor. 2.5]{MMQRS} as follows:

\begin{theorem}
	Let $\mathbb{X}$ be a Banach space and let $x \in S_\mathbb{X}.$ Suppose $C \subseteq \mathbb{X}^*$ is such that $B_{\mathbb{X}^*} = \overline{co(C)}^{w^*}.$ Then for any $y \in \mathbb{X},$ $x \perp_B y$ if and only if $0 \in co(\{\lim x_n^*(y) : x_n^* \in C \, \mbox{and}\, \lim x_n^*(x) =1 \}).$
\end{theorem}

Following the same technique, a generalized version of Lemma \ref{Chem} can be derived as follows:
\begin{theorem}\label{V:set}
	Let $\mathbb{X}$ be a Banach space and let $ x, y \in S_{\mathbb{X}}$.  Suppose  $ C \subset B_{\mathbb{X}^*}$ is  such that $B_{\mathbb{X}^*}= \overline{co(C)}^{w^*}.$
	Then the following are equivalent:
	\begin{itemize}
		\item[(i)] $x \perp_B^\epsilon y$
		\item[(ii)] $co (\{\lim x_n^*(y): x_n^* \in C ~\forall n \in \mathbb{N}, \lim x_n^*(x) = 1	\}) \cap \mathcal{D}(\epsilon) \neq \emptyset. $
	\end{itemize}
\end{theorem} 
\begin{proof}
	From Lemma \ref{Chem}, $x \perp_B^\epsilon y$ if and only if there exists $x^* \in J(x)$ such that $x^*(y) \in \mathcal{D}(\epsilon).$ Clearly, this is equivalent to $V(\mathbb{X}, x, y) \cap \mathcal{D}(\epsilon) \neq \emptyset,$ where $V(\mathbb{X}, x, y)= \{\phi(y): \phi \in S_{\mathbb{X}^*}, \phi(x)=1 \}.$ Now from   \cite[Th. 2.3]{MMQRS} we note that 
	$$V(\mathbb{X}, x, y)   = co ( \{\lim x_n^*(y): x_n^* \in C ~\forall n \in \mathbb{N}, \lim x_n^*(x) = 1	\}),$$
	thereby finishing the proof.
\end{proof}

Applying the above result, it is rather easy to characterize the approximate Birkhoff-James orthogonality in $\ell_\infty(K, \mathbb{X})$ and in $C_0(K, \mathbb{X}).$ Let us first note the following characterization of Birkhoff-James orthogonality.
\begin{theorem}\label{martin1} \cite[Th. 3.2]{MMQRS}
	Let $K$ be a non-empty set and let $\mathbb{X}$ be a Banach space. Let $C \subset S_{\mathbb{X}^*}$ be such that $B_{\mathbb{X}^*}= \overline{co(C)}^{w^*}.$  Suppose that $f, g \in \ell_\infty(K, \mathbb{X}).$ Then $	f \perp_B g $ if and only if 
	\[
	0 \in 	co\bigg(\bigg\{ \lim y_n^*(g(k_n)): k_n \in K, y_n^* \in C, ~\forall n \in \mathbb{N},   \lim y_n^*(f(k_n))= \|f\|\bigg\}\bigg) .
	\]
\end{theorem}

\begin{theorem}\label{martin:continuous}\cite[Th. 3.5]{MMQRS}
	Let $K$ be a locally compact Hausdorff space and let $\mathbb{X}$ be a Banach space. Suppose   $C \subset S_{\mathbb{X}^*}$ is such that $B_{\mathbb{X}^*}= \overline{co(C)}^{w^*}.$  Then  for $f, g \in C_0(K, \mathbb{X} ),$   
	\[ f \perp_B g \iff	0 \in co\bigg(\bigg\{  y^*(g(k)): k \in K, y^* \in C,   y^*(f(k))= \|f\|\bigg\}\bigg).
	\]
\end{theorem}

Although in \cite[Th. 3.5]{MMQRS}, the characterization has been given in $C(K, \mathbb{X})$ space considering $K$ to be compact Hausdorff, we can replicate the arguments for the space $C_0(K, \mathbb{X}),$ where $K$ is a locally compact Hausdorff space. In this context, the crucial thing to observe is that for any $f \in C_0(K, \mathbb{X}),$ $M_f$ is non-empty and compact.\\

Using similar technique as in Theorem $ 3.2 $ and Theorem $ 3.5 $ of \cite{MMQRS}, and using Theorem \ref{V:set}, we obtain the following characterizations of approximate Birkhoff-James orthogonality:

\begin{theorem}\label{approximate:martin1}
	Let $K$ be a non-empty set and let $\mathbb{X}$ be a Banach space. Suppose that $C \subset S_{\mathbb{X}^*}$ is such that $B_{\mathbb{X}^*}= \overline{co(C)}^{w^*}.$  Then  for $f, g \in \ell_\infty(K, \mathbb{X}),$  $	f \perp_B^\epsilon g $ if and only if 
	\[
	co\bigg(\bigg\{ \lim y_n^*(g(k_n)): k_n \in K, y_n^* \in C, ~\forall n \in \mathbb{N},   \lim y_n^*(f(k_n))= \|f\|\bigg\}\bigg) \cap \mathcal{D}(\epsilon\|g\|) \neq \emptyset.
	\]
\end{theorem}

\begin{theorem}\label{approximate:continuous}
	Let $K$ be a locally compact Hausdorff space and let $\mathbb{X}$ be a Banach space. Suppose  $C \subset S_{\mathbb{X}^*}$ is such that $B_{\mathbb{X}^*}= \overline{co(C)}^{w^*}.$  Then  for $f, g \in C_0(K, \mathbb{X} ),$
	\[
	f \perp_B^\epsilon g \iff
	co\bigg(\bigg\{  y^*(g(k)): k \in K, y^* \in C,   y^*(f(k))= \|f\|\bigg\}\bigg) \cap \mathcal{D}(\epsilon\|g\|) \neq \emptyset.
	\]
\end{theorem}


Applying Theorem \ref{approximate:martin1}, we characterize the approximate Birkhoff-James orthogonality in the space $\ell_\infty(K).$ 

\begin{theorem}\label{approximate:l(K)}
	Let $K$ be a non-empty set and let $f,g \in S_{\ell_\infty(K)}.$ Then $f \perp_B^\epsilon g$ if and only if 
	\[
	 co\bigg(\bigg\{\lim \overline{\gamma_n} g(k_n): k_n \in K, \gamma_n \in \mathbb{K}, |\gamma_n|=1 ~ \forall n \in \mathbb{N}, \lim \gamma_n f(k_n)=1\bigg\}\bigg) \cap \mathcal{D}(\epsilon) \neq \emptyset.
	\] 
	In particular, $f \perp_B g$ if and only if 
	\[
0 \in 	co\bigg(\bigg\{\lim \overline{\gamma_n} g(k_n): k_n \in K, \gamma_n \in \mathbb{K}, |\gamma_n|=1 ~  \forall n \in \mathbb{N}, \lim \gamma_n f(k_n)=1\bigg\}\bigg).
	\] 
\end{theorem}


As $\mathbb{L}(\mathbb{X}, \mathbb{Y})$ can be embedded into $\ell_\infty(S_{\mathbb{X}}, \mathbb{Y}),$ using Theorem \ref{approximate:martin1} we obtain the following characterization. 

\begin{theorem}\label{approximate:operator}
	Let $\mathbb{X}, \mathbb{Y}$ be two Banach spaces and let $ T , A \in S_{\mathbb{L}(\mathbb{X}, \mathbb{Y})}.$ Suppose that $\epsilon\in [0,1).$  Then $T \perp_B^\epsilon A$ if and only if 
	$$	co \bigg(\bigg\{\lim y_n^*(Ax_n): (x_n, y_n^*)\in S_\mathbb{X}\times S_{\mathbb{Y}^*}, \lim y_n^*(Tx_n)=1\bigg\}\bigg) \cap \mathcal{D}(\epsilon) \neq \emptyset.$$
\end{theorem}

We end this section with the following classical result from point-set topology which will be used repeatedly in the later parts. 

\begin{lemma}(Uryshon's lemma) \cite{Munkres}\label{Uryshon}
Let $K$ be a normal space. Let $A$ and $B$ be disjoint closed subsets of $K$. Then there exists a continuous map $f : K \to [0, 1] $ such that $f(x)= 0$  for every $x \in A,$ and $f(x) = 1$ for every $x \in B.$
\end{lemma}

\section{main result}

We begin this section with the following lemma, which we require for characterizing anti-coproximinal subspaces.

\begin{lemma}\label{Rudin}\cite[Th. 4.7]{R}
	Let $\mathbb{Z}$ be a subspace of $\mathbb{X}^*.$ Then $(^\perp \mathbb{Z})^\perp= \overline{\mathbb{Z}}^{w^*},$ where $^\perp \mathbb{Z}= \{ x \in \mathbb{X}: z^*(x)=0~ \forall z^* \in \mathbb{Z}\}$ and $ (^\perp \mathbb{Z})^\perp = \{x^* \in \mathbb{X}^*: x^*(z)=0~ \forall z \in \, ^{\perp}\mathbb{Z}\}.$
\end{lemma}

\begin{prop}\label{sufficient}
	Let  $\mathbb{Y}$ be a closed proper subspace of a Banach space  $\mathbb{X}.$ Then $\mathbb{Y}$ is anti-coproximinal in $\mathbb{X}$ if  $ \overline{span~\mathcal{J}_\mathbb{Y}}^{w^*} = \mathbb{X}^*.$ \\
	Moreover, when $Sm(\mathbb{X}) \cap \mathbb{Y}$ is dense in $\mathbb{Y},$ then  $\mathbb{Y}$ is anti-coproximinal in $\mathbb{X}$ if and only if $ \overline{span\mathcal{J}_\mathbb{Y}}^{w^*} = \mathbb{X}^*.$ 
\end{prop}

\begin{proof}
	Suppose on the contrary  that $\mathbb{Y}$ is not anti-coproximinal in $\mathbb{X}.$ Then there exists an element $x \in \mathbb{X}\setminus \mathbb{Y}$ such that $\mathbb{Y} \perp_B x.$  Let $x^* \in \mathbb{X}^*.$ Since $\overline{span~\mathcal{J}_\mathbb{Y}}^{w^*} = \mathbb{X}^*,$ it follows that there exists a net $\{x_\alpha^*\}_{\alpha \in \Lambda} \in span~ \mathcal{J}_\mathbb{Y}$ such that $x_\alpha^* \overset{w^*}{\longrightarrow} x^*.$ Since $\mathbb{Y} \perp_B x,$ for any $y\in Sm(\mathbb{X}) \cap \mathbb{Y},$ we have $y \perp_B x.$ Using Lemma \ref{James}, we obtain that for any $y^* \in \mathcal{J}_\mathbb{Y},$ $y^*(x)=0$ and therefore, $x_\alpha^*(x)=0,$ for each $\alpha \in \Lambda.$ So, $x^*(x)=0.$ As $x^*$ is taken arbitrarily from $\mathbb{X}^*,$ we obtain that $x=0,$ a contradiction.\\
	Let us now assume that $Sm(\mathbb{X}) \cap \mathbb{Y}$ is dense in $\mathbb{Y}.$ We only need to prove the necessary part.
	Suppose on the contrary that $\overline{span \mathcal{J}_\mathbb{Y}}^{w^*} \subsetneqq \mathbb{X}^*.$  Clearly, $^{\perp}span\mathcal{J}_\mathbb{Y} = \cap_{y^* \in \mathcal{J}_\mathbb{Y}} \ker y^*.$ Therefore, following Lemma \ref{Rudin}, we conclude that $(\cap_{y^*\in \mathcal{J}_\mathbb{Y}} \ker y^*)^{\perp}=\overline{span\mathcal{J}_\mathbb{Y}}^{w^*}.$ Thus, $$(\cap_{y^* \in \mathcal{J}_\mathbb{Y}} \ker y^*)^*= \mathbb{X}^*/\overline{span\mathcal{J}_\mathbb{Y}}^{w^*}, $$ which implies that $\cap_{y^* \in \mathcal{J}_\mathbb{Y}} \ker y^* \neq 0.$ Let $z \in \cap_{y^* \in \mathcal{J}_\mathbb{Y}} \ker y^*$ and let  $y \in \mathbb{Y}$ be arbitrary. Since $Sm(\mathbb{X})\cap \mathbb{Y}$ is dense in $\mathbb{Y},$  there exists a sequence $\{y_n\} \subset Sm(\mathbb{X}) \cap \mathbb{Y}$ such that $y_n \longrightarrow y.$ Note that for each $n \in \mathbb{N},$ $y_n^* \in \mathcal{J}_\mathbb{Y},$ where $J(y_n) = \{y_n^*\}.$  As $z \in \cap_{y^* \in \mathcal{J}_\mathbb{Y}} \ker y^*,$ we have $y_n \perp_B z,$ for each $n \in \mathbb{N}.$ Since $y_n \longrightarrow y,$ it follows that $y \perp_B z.$ This implies that $\mathbb{Y} \perp_B z,$  contradicting the hypothesis that $\mathbb{Y}$ is anti-coproximinal in $\mathbb{X}$.
\end{proof}

We now give an example which illustrate the applicability of Proposition \ref{sufficient}.

\begin{example}
	Let us consider the space $\ell_p,$ where $ p \in (1, \infty) \setminus \{ 2 \} $ and $n \geq 3$. Let $e_i= (0, 0, \ldots, 0, \underset{i-th~position}{1}, 0, \ldots)  ~\forall i \in \mathbb{N}.$ Let $\phi$ be the canonical isometric isomorphism from $(\ell_p)^*$ to $\ell_q.$  We denote $e'_i \in (\ell_p)^*$ as $e'_i(x_1, x_2, \ldots) = x_i~ \forall i \in \mathbb{N}.$
    Suppose that  $\mathbb{Y} = span\{\widetilde{x}_1, \widetilde{x}_2, e_3, e_4, \ldots \},$  $\widetilde{x}_1= (1, 1, 1, 0, 0, \ldots), \widetilde{x}_2 = (1, 2, 3, 0, 0,\ldots). $ Let $J(\widetilde{x}_i)= \{\widetilde{f}_i\},$ for $i=1, 2.$  It is straightforward computation to verify that \\
	$ \phi(\widetilde{f}_1) = \bigg(\frac{1}{3^{1-\frac{1}{p}}}, \frac{1}{3^{1-\frac{1}{p}}}, \frac{1}{3^{1-\frac{1}{p}}}, 0, 0, \ldots \bigg) \in \ell_q,$\\
	$ \phi(\widetilde{f}_2) = \bigg(\frac{1}{(1+2^p+3^p)^{1-\frac{1}{p}}}, \frac{2^{p-1}}{(1+2^p+3^p)^{1-\frac{1}{p}}}, \frac{3^{p-1}}{(1+2^p+3^p)^{1-\frac{1}{p}}}, 0, 0, \ldots \bigg) \in \ell_q.$\\
	Consider the element $ \widetilde{x} = 3\widetilde{x}_1 - \widetilde{x}_2 \in \mathbb{Y}$ and let $ J( \widetilde{x})= \{ \widetilde{f}\}.$ Again  we get $ \phi (\widetilde{f}) = \Big(\frac{2^{p-1}}{(2^p+1)^{1-\frac{1}{p}}}, \frac{1}{(2^p+1)^{1-\frac{1}{p}}}, 0, 0,  \ldots\Big) \in \ell_q.$ Observe that $\{\phi(\widetilde{f}_1), \phi(\widetilde{f}_2), \phi(\widetilde{f})\} $ is a linearly independent set in $\ell_q$.  Since $\phi(\widetilde{f}_1), \phi(\widetilde{f}_2), \phi(\widetilde{f}) \in \mathcal{J}_{\mathbb{Y}},$ it follows that  $e'_1, e'_2, e'_3 \in span~\mathcal{J}_{\mathbb{Y}}.$ Moreover, as for any $k > 3$, $e_k \in \mathbb{Y}$ and $J(e_k)=\{e_k'\},$ we have $e_k' \in \mathcal{J}_\mathbb{Y}.$ Therefore, for any $ k \in \mathbb{N},$ $e_k' \in span~\mathcal{J}_{\mathbb{Y}}.$ This implies $span~\{e_k': k \in \mathbb{N} \} \subset span~\mathcal{J}_{\mathbb{Y}}.$ Therefore, $$(\ell_p)^* = \overline{span~\{e_k': k \in \mathbb{N}\}} \subseteq \overline{span~\{e_k': k \in \mathbb{N}\}}^{w^*} \subseteq  \overline{  span~\mathcal{J}_{\mathbb{Y}}}^{w^*} \subseteq (\ell_p)^* .$$
    So, $ \overline{  span~\mathcal{J}_{\mathbb{Y}}}^{w^*} = (\ell_p)^*$ and applying Theorem \ref{sufficient}, $\mathbb{Y}$ is an anti-coproximinal subspace in $\mathbb{X}$.\\
	
\end{example}


One of the main aims of this article is to illustrate the geometric specialty of strongly anti-coproximinal subspaces of a Banach space. We show this by deriving a necessary condition for a subspace to be strongly anti-coproximinal.

\begin{theorem}\label{ALUR}
	Let  $\mathbb{Y}$ be a closed proper subspace of a Banach space  $\mathbb{X}.$ Suppose that $ S_{\mathbb{X}} \setminus S_\mathbb{Y}$ contains an w-ALUR point. Then $\mathbb{Y}$ is not strongly anti-coproximinal in $\mathbb{X}$.
\end{theorem}

\begin{proof}
	Let $x \in S_{\mathbb{X}} \setminus S_\mathbb{Y}$ be an w-ALUR point.
	Suppose on the contrary that $\mathbb{Y}$ is strongly anti-coproximinal in $\mathbb{X}$. Thus for any $\epsilon \in [0, 1),$ $\mathbb{Y} \not\perp_B^{\epsilon} x$.  Let us take $\{\epsilon_n\}_{n \in \mathbb{N}}\subset [0, 1)$ such that $\epsilon_n \longrightarrow 1.$ For each $n \in \mathbb{N},$ there exists $y_n \in S_\mathbb{Y}$ such that  $y_n \not \perp_B^{\epsilon_n} x.$ Let $y_n^* \in S_{\mathbb{X}^*}$ be such that $y_n^*(y_n)=1.$ 	Since $B_{\mathbb{X}^*}$ is weak*-compact,
	it follows that there exists a weak*-cluster point $y^* \in B_{\mathbb{X}^*}$ of the sequence $\{y_n^*\}_{n\in \mathbb{N}}.$
	This implies $y^*(x)$ is a cluster point of $\{y_n^*(x)\}. $ Following \cite[ Th. 8 (p. 72)]{kelly}, there exists a subsequence $\{y_{n_k}^*(x)\}$ of $\{y_n^*(x)\}$ such that $y_{n_k}^*(x) \to y^*(x).$ Now from Lemma \ref{Chem}, $y_{n_k} \not\perp_B^{\epsilon_{n_k}} x$ implies that $|y_{n_k}^*(x)| \geq \epsilon_{n_k}\|x\|,$ for each $k \in \mathbb{N}.$ So, $|y^*(x)| \geq \|x\|,$ which in turn implies that $|y^*(x)|= \|x\|.$
	 This implies $y^*\in J(x) $ or $y^* \in J(-x).$ Suppose that $y^* \in J(x).$ Note that
	$$ y_{n_k}^* \bigg(\frac{y_{n_k}+x}{2}\bigg)=  \frac{y_{n_k}^*(y_{n_k})}{2} + \frac{y_{n_k}^*(x)}{2}= \frac{1}{2}+\frac{1}{2} y_{n_k}^*(x) ,$$ 
and  so $\lim  y_{n_k}^* (\frac{y_{n_k}+x}{2})=1.$ Also,
   $x$ being an w-ALUR point, we get that $y_{n_k} \overset{w}{\longrightarrow} x.$ As $\mathbb{Y}$ is a closed subspace of $\mathbb{X},$ so $x \in \mathbb{Y}.$ This contradicts the fact that $x \in S_\mathbb{X} \setminus S_\mathbb{Y}.$ If $y^* \in J(-x)$ then proceeding similarly we get $\lim y^*(\frac{y_{n_k} - x}{2}) = 1$ and consequently, $-x \in \mathbb{Y},$ again a contradiction. Thus  $\mathbb{Y}$ is strongly anti-coproximinal in $\mathbb{X}$. 
\end{proof}

\begin{remark}
    For a strictly convex Banach space $\mathbb{X},$ each $x \in S_{\mathbb{X}}$  is a w-ALUR point. Therefore, applying Theorem \ref{ALUR}, we see that there does not  exist any closed strongly anti-coproximinal subspace in $\mathbb{X}$. In particular, the above result tells that  if the set of all w-ALUR points of $\mathbb{X}$ is nonempty,  then a closed strongly-anticoproximinal subspace $\mathbb{Y}$ of $\mathbb{X}$ contains all w-ALUR points of $\mathbb{X}.$  On the other hand, if $\mathbb{X}$ has no w-ALUR point then it might contain  strongly anti-coproximinal subspaces.  For  an example consider $\mathbb{X}= \ell_\infty^3.$  Following \cite[Th. 2.25]{SSGP3}, it is strightforward to see that the subspace   $\mathbb{Y} = span \{ (3, 0, 2), ( 0, 3, 2)\}$  is strongly anti-coproximinal in $\mathbb{X}$.
\end{remark}

The following corollary is immediate from Theorem \ref{ALUR}.

\begin{cor}
	Let $\mathbb{Y}$ be a closed proper subspace of a Banach space $\mathbb{X}.$ Suppose that the set of all w-ALUR points of $B_{\mathbb{X}}$ separates $\mathbb{X}^*.$ Then $\mathbb{Y}$ is not  strongly anti-coproximinal in $\mathbb{X}$.
   \end{cor}

   \begin{proof}
	Consider  $x^* \in \mathbb{X}^*$ such that $\mathbb{Y} \subseteq \ker x^*.$
	Suppose that the set of all w-ALUR points of  $B_{\mathbb{X}}$ separates $\mathbb{X}^*.$ Therefore, there exists a w-ALUR point $x \in S_\mathbb{X}$ such that $x^*(x) \neq 0.$ Since $\mathbb{Y} \subseteq \ker x^*,$ it follows that $x \notin \mathbb{Y}.$  Thus $x \in S_\mathbb{X} \setminus S_{\mathbb{Y}}$ is a  w-ALUR point and therefore, applying Theorem \ref{ALUR}, we get the desired result.
\end{proof}

   In the next theorem, we give another  necessary condition for  strongly anti-coproximinal finite-dimensional subspaces, further illustrating its geometric specialty.

\begin{theorem}\label{theorem:face}
	Let  $\mathbb{Y}$ be a  strongly anti-coproximinal finite-dimensional subspace of a Banach space $\mathbb{X}.$ Suppose that $F$ is a  maximal face of $B_{\mathbb{X}}$ and $x \in int(F).$ Then there exists $y \in \mathbb{Y}$ such that    $J(x) \cap J(y) \neq \emptyset.$ 
\end{theorem}

\begin{proof}
 As $\mathbb{Y}$ is strongly anti-coproximinal in $\mathbb{X}$, we have $\mathbb{Y} \not\perp_B^\epsilon x,$ for all $\epsilon\in [0, 1).$  Let us take $\{\epsilon_n\}_{n \in \mathbb{N}}\subset [0, 1)$ such that $\epsilon_n \longrightarrow 1.$ For each $n \in \mathbb{N},$ there exists $y_n \in S_\mathbb{Y}$ such that  $y_n \not \perp_B^{\epsilon_n} x.$ Let $y_n^* \in S_{\mathbb{X}^*}$ be such that $y_n^*(y_n)=1.$ 	Since $B_{\mathbb{X}^*}$ is weak*-compact,
 it follows that there exists a weak*-cluster point $y^* \in B_{\mathbb{X}^*}$ of the sequence $\{y_n^*\}_{n\in \mathbb{N}}.$
 This implies $y^*(x)$ is a cluster point of $\{y_n^*(x)\}. $ Following \cite[ Th. 8 (p. 72)]{kelly}, there exists a subsequence $\{y_{n_k}^*(x)\}$ of $\{y_n^*(x)\}$ such that $y_{n_k}^*(x) \to y^*(x).$ Now from Lemma \ref{Chem}, $y_{n_k} \not\perp_B^{\epsilon_{n_k}} x$ implies that $|y_{n_k}^*(x)| \geq \epsilon_{n_k}\|x\|,$ for each $k \in \mathbb{N}.$ So, $|y^*(x)| \geq \|x\|,$ which in turn implies that $|y^*(x)|= \|x\|.$
 This implies $y^*\in J(x) $ or $y^* \in J(-x).$
  Let $y^* \in J(x).$ Since $\mathbb{Y}$ is finite-dimensional, it follows that $S_{\mathbb{Y}}$ is compact and therefore, $y_{n_k} \longrightarrow y,$ for some $y \in S_\mathbb{Y}.$ So for each $k \in \mathbb{N},$
 \begin{eqnarray*}
 	|y_{n_k}^*(y_{n_k})-y^*(y)|&=&|y_{n_k}^*(y_{n_k}) - y_{n_k}^*(y) + y_{n_k}^*(y) - y^*(y)|\\
 	&\leq& \|y_{n_k}^*\|\|y_{n_k}-y\| + |y_{n_k}^*(y)-y^*(y)|.
 \end{eqnarray*}
 Taking $k \longrightarrow \infty$ in the above relation we get  $y_{n_k}^*(y_{n_k}) \longrightarrow y^*(y).$ Since for each $ k,$ $y_{n_k}^*(y_{n_k})=1,$ we have $y^*(y)=1.$ In other words, $y^* \in J(y).$ Thus we get $y^* \in J(x) \cap J(y).$ On the other hand, if $y^* \in J(-x)$ then $-y^*\in J(x).$ Then proceeding similarly as above, we can conclude that $J(x) \cap J(-y) \neq \emptyset.$ This completes the proof.

\end{proof}	

The following observation, also geometric in nature, discusses the intersection property of finite-dimensional strongly anti-coproximinal subspaces of a Banach space.

	\begin{theorem}\label{max face}
		Let $\mathbb{X}$ be a Banach space and let $\mathbb{Y}$ be a finite-dimensional subspace of $\mathbb{X}$.
	 If $\mathbb{Y}$ is strongly anti-coproximinal in $\mathbb{X},$ then $\mathbb{Y}$
intersects every maximal face of $B_\mathbb{X}.$
	\end{theorem}
	
	\begin{proof}
	 Let $F$ be a maximal face of $B_{\mathbb{X}}$ and let $x \in int(F).$ From Theorem \ref{theorem:face},	we get an element $y \in \mathbb{Y}$ such that $J(x) \cap J(y) \neq \emptyset.$ Suppose that  $z^* \in J(x) \cap J(y),$ for some $y \in \mathbb{Y}.$ As $x \in int(F)$ and $z^*(x)=1,$ it is easy to check that for any $v \in F,$ $z^*(v)=1.$
	Let $w \in co(F \cup \{y\}).$ Then $w= (1-t) z+ ty,$ for some $t \in [0,1]$ and $z \in F.$ The relation $z^*(y)=z^*(z)=1$ yields that $z^*(w)=1.$ This implies that $w \in S_{\mathbb{X}}.$ So, $co(F \cup \{y\}) \subset S_{\mathbb{X}}.$ Therefore, there exists some face $F' $ of $B_{\mathbb{X}}$ such that $F \subset co(F \cup \{y\}) \subset F'.$ As $F$ is a maximal face, we have $F=co(F \cup \{y\}),$ so $y \in F.$ Therefore, $\mathbb{Y} \cap F \neq \emptyset,$ as desired.
	\end{proof}
	
\begin{remark}

 It is easy to verify that whenever $x \in S_{\mathbb{X}}$ is a w-ALUR, $\{x\}$ is a maximal face. Following Theorem \ref{ALUR}, for a closed proper subspace $\mathbb{Y}$ to be strongly anti-coproximinal in $\mathbb{X}$, $\mathbb{Y}$ has to intersect every maximal face of the form $\{x\},$ where $x \in S_{\mathbb{X}}$ is  w-ALUR. On the other hand, Theorem \ref{max face} implies that a finite-dimensional strongly anti-coproximinal subspace  of $\mathbb{X}$ intersects every maximal face of $B_{\mathbb{X}}$.
\end{remark}
Applying Theorem \ref{max face}, we give an example of a Banach space having no strongly anti-coproximinal subspace.
\begin{example}
  	Let $\mathbb{X}$ be the $3$-dimensional Banach space whose unit ball is a prism-pyramid with hexagon base  such that  
    $Ext(B_{\mathbb{X}})= \{ \pm (1,0,1), \pm (\frac{1}{2}, \frac{\sqrt{3}}{2}, 1),   \pm  (-\frac{1}{2}, \frac{\sqrt{3}}{2}, 1), $ $ \pm (-1, 0, 1), \pm (-\frac{1}{2}, - \frac{\sqrt{3}}{2}, 1), \pm (\frac{1}{2}, -\frac{\sqrt{3}}{2}, 1) , \pm (0,0,2)\}$ (see \cite[Th. 2.5 (fig. 3)]{SPBB}). From the structure of the unit ball of $\mathbb{X}$, it can be observed that there does not exist any  $2$-dimensional subspace  of $\mathbb{X}$ which intersects all the maximal faces of $B_\mathbb{X}$. Therefore, applying Theorem \ref{max face}, there exists no strongly anti-coproximinal subspace in $\mathbb{X}.$ Note that $\mathbb{X}$ has no w-ALUR points. Therefore, we are unable to apply Theorem \ref{ALUR} in this case. However, Theorem \ref{max face} ensures that there exists no strongly anti-coproximinal subspace of $\mathbb{X}.$
    \end{example}

	As a consequence of Theorem \ref{max face}, the following corollary is immediate.
	
	\begin{cor}
		Let $\mathbb{Y}$ be a finite-dimensional proper subspace of $\mathbb{X}.$ Suppose that  the set of all rotund points of unit ball of $B_{\mathbb{X}}$ separates $\mathbb{X}^*.$ Then $\mathbb{Y}$ is not a strongly anti-coproximinal subspace of $\mathbb{X}.$
       \end{cor} 
    
         For a finite-dimensional subspace $\mathbb{Y}$ of $\mathbb{X},$ if $S_\mathbb{X} \setminus S_\mathbb{Y}$ contains a rotund point then $\mathbb{Y}$ is not strongly anti-coproximinal subspace in $\mathbb{X}.$


Our next goal is to show that the only real Banach spaces which contain a finite-dimensional strongly anti-coproximinal subspace must be themselves finite-dimensional polyhedral. We require the following lemma for our purpose.

\begin{lemma}\label{nu index}\cite[Lemma 2.1]{SPBB}
	Let $\mathbb{X}$ be an $n$-dimensional polyhedral Banach space. Then $f \in S_{\mathbb{X}^*}$ is an extreme point of $B_{\mathbb{X}^*}$ if and only if $f$ is a supporting functional corresponding to a maximal face of $B_\mathbb{X}.$
\end{lemma}
 
 \begin{theorem}
 	Let $\mathbb{Y}$ be a finite-dimensional polyhedral subspace of a real Banach space $\mathbb{X}.$ Suppose $\mathbb{Y}$ is strongly anti-coproximinal in $\mathbb{X}$. Then 
 	\begin{itemize}
 		\item [(i)] for any two distinct maximal faces $F_1, F_2$ of $B_{\mathbb{X}},$ $F_1 \cap \mathbb{Y} \neq F_2 \cap \mathbb{Y}$
 		
 \item [(ii)] $\mathbb{X}$ is a finite-dimensional polyhedral Banach space.
 	\end{itemize}
 \end{theorem}
 
 \begin{proof}
 	(i) If $F$ is a face of $B_{\mathbb{X}},$ then it clear that $F \cap B_{\mathbb{Y}}$ is also a face. Suppose on the contrary that $F_1, F_2$ are two distinct maximal faces of $B_{\mathbb{X}}$ such that $F_1 \cap \mathbb{Y}= F_2 \cap \mathbb{Y}.$ Let $ \pm Q_1, \pm Q_2, \ldots, \pm Q_r$ be the distinct maximal faces of $B_\mathbb{Y}.$   Suppose that for each  $1 \leq i \leq r,$ $y_i^* \in S_{\mathbb{Y}^*}$ supports the face $Q_i.$  Observe that this correspondence between $y_i^*$ and $Q_i$ is one-one (see Lemma \ref{nu index}). For each  $1 \leq i \leq r,$ suppose that $x_i^* \in S_{\mathbb{X}^*}$ is a Hahn-Banach extension of $y_i^*.$  Take $x \in int(F_1).$ We claim that $x_k^*$ supports the face $F_1,$ for some $k,  1\leq k \leq r.$ Otherwise, let us assume that $ x_i^*$ do not support the face $F_1,$ for any $1 \leq i \leq r.$ Let $ \max \{ |x_i^*(x)|: 1 \leq i \leq r\}=\epsilon_0.$ Clearly $\epsilon_0 < 1.$ Let $y \in S_{\mathbb{Y}}.$ Then $ y \in Q_k,$ for some $1 \leq k \leq r.$ So, $x_k^* \in J(y)$ and $|x_k^*(x)| \leq \epsilon_0.$ Thus from Lemma \ref{Chem}, we get $y \perp_B^{\epsilon_0} x.$ Therefore, $\mathbb{Y} \perp_B^{\epsilon_0} x,$ a contradiction to the fact that $\mathbb{Y}$ is strongly anti-coproximinal in $\mathbb{X}$.This establishes our claim and so $x_k^*$ supports the face $F_1,$ for some $k, 1 \leq k \leq r.$  As $F_1$ is a maximal face, observe that $Q_k \subset F_1.$  Let $z^* \in S_{\mathbb{X}^*}$ support the face $F_2.$ Clearly, $z^* \neq \pm x_k^*.$ Let $$\epsilon= \max \{|z^*(x)|, |x_i^*(x)|: i \in \{1, 2, \ldots,r\} \setminus \{k\}\}.$$
 	We claim  that $\epsilon <1.$ First we note that $|z^*(x)| < 1,$ as otherwise, $F_1= F_2.$ If possible, let $j \in \{1, 2, \ldots, r\} \setminus \{k\}$ be such that $x_j^*(x)=1.$ Then $x_j^*$ supports $F_1.$  It is clear that $x_j^*$ and $x_k^*$ support $Q_j$ and $Q_k,$ respectively. Since $x_j^*, x_k^*$ support the maximal face $F_1,$ it is easy to observe that $Q_j, Q_k \subset F_1 \cap B_\mathbb{Y}.$ This contradicts $Q_j, Q_k$ are two distinct  maximal faces of $B_{\mathbb{Y}}.$ Therefore, $\epsilon < 1.$ Now take $y \in S_\mathbb{Y}.$ If $y \in Q_j,$ for some $j \in \{1, 2, \ldots, r\}\setminus \{k\}$ then $x_j^* \in J(y)$ and $|x_j^*(x)| \leq \epsilon.$ Suppose that $y \in Q_k.$ Since  $ Q_k \subset F_1 \cap \mathbb{Y} = F_2 \cap \mathbb{Y},$ it follows that $z^* \in J(y),$ and so, we again obtain that $|z^*(x)| \leq \epsilon.$ This proves that  $y \perp_B^{\epsilon} x.$ Since  $y\in S_\mathbb{Y}$ is arbitrary, $\mathbb{Y} \perp_B^{\epsilon} x,$ contradicting the fact that $\mathbb{Y}$ is strongly anti-coproximinal in $\mathbb{X}$.\\

 	(ii)  Suppose that $ \pm Q_1, \pm Q_2, \ldots, \pm Q_r$ are the maximal faces of $B_{\mathbb{Y}}.$ Clearly $Q_i= F_i \cap B_{\mathbb{Y}},$ for some maximal face $F_i$ of $B_{\mathbb{X}}.$ We claim that $\pm F_1, \pm F_2, \ldots, \pm F_r$ are the only maximal faces of $B_{\mathbb{X}}.$ If possible, let $F$ be a maximal face of $B_{\mathbb{X}}$ such that $F \neq \pm F_i,$ for all $i =1, 2, \ldots, r.$ From Theorem \ref{max face}, we note that  $F \cap B_{\mathbb{Y}}\neq \emptyset.$ Clearly, $F \cap B_\mathbb{Y}$ is a maximal face of $B_\mathbb{Y}.$ Therefore, $F \cap B_\mathbb{Y} = Q_j = F_j \cap B_\mathbb{Y},$ for some $j \in \{1, 2, \ldots, r\}.$ Since $F \neq \pm F_j,$ following (i) we arrive at a contradiction. This proves our claim and consequently, $\mathbb{X}$ is a finite-dimensional polyhedral Banach space.

 	 	 \end{proof}

\begin{remark}
	We observe that the condition given in Theorem \ref{max face} is necessary but not sufficient. Consider the space $\mathbb{X} = \ell_1^3(\mathbb{R})$ and the subspace  $\mathbb{Y}= span \{(1, 0, 0), (0, 1, 0)\}.$ Then it is easy to see that the subspace $\mathbb{Y}$ intersects all the maximal faces of $B_\mathbb{X},$ whereas $\mathbb{Y}$ is a coproximinal subspace. Indeed, given any $(x, y, z) \in \mathbb{X} \setminus \mathbb{Y},$ $(x, y, 0) \in \mathbb{Y}$ is a best coapproximation to $(x, y, z)$ out of $\mathbb{Y}.$ This shows that the subspace $\mathbb{Y}$ may intersect every maximal face of $B_\mathbb{X}$ but is not strongly anti-coproximinal in $\mathbb{X}$. However, as we observe in the next result, if $\mathbb{Y}$ intersects the relative interior of every maximal face of $B_\mathbb{X}$, then $\mathbb{Y}$ is indeed strongly anti-coproximinal in $\mathbb{X}$.
\end{remark}

\begin{theorem}\label{sufficient:interior}
	Let $\mathbb{X}$ be a Banach space and let $\mathbb{Y}$ be a    proper subspace of $\mathbb{X}.$ If $\mathbb{Y}$ intersects the relative interior of every facet of $B_\mathbb{X}$ then $\mathbb{Y}$ is strongly anti-coproximinal in $\mathbb{X}$.
\end{theorem}

\begin{proof}
	Suppose on the contrary that $\mathbb{Y}$ is not strongly anti-coproximinal in $\mathbb{X}$. Then there exists $x \in S_\mathbb{X} \setminus S_\mathbb{Y}$ and $\epsilon \in [0, 1)$ such that $\mathbb{Y} \perp_B^{\epsilon} x.$ Let us consider a maximal face $F \subset S_\mathbb{X}$ such that $x \in F.$ Since $\mathbb{Y} \cap int(F)\neq \emptyset,$ let us take $y \in S_\mathbb{Y} \cap int(F).$  Then there exists $z \in F$ such that $y=(1-t)x+ tz,$ for some $t \in (0, 1).$ It is immediate that for any $y^* \in J(y),$ we get $y^*(x)=1.$ This implies $y \not\perp_B^{\epsilon} x,$ a contradiction. This establishes the theorem.
\end{proof}

 We end this section with the following remark.

 \begin{remark}
We would like to mention that we could not establish whether the above-mentioned sufficient condition for strongly anti-coproximinal subspaces is also necessary. However, in case of a \textit{finite-dimensional polyhedral Banach space $ \mathbb{X}, $} it is known that a subspace $\mathbb{Y}$ of $ \mathbb{X} $ is strongly anti-coproximinal if and only if  $\mathbb{Y}$ intersects the relative interior of every facet of $B_\mathbb{X}$ (see \cite[Th. 2.20]{SSGP3}). 
 \end{remark}

\section*{Anti-coproximinality in function spaces}

In this section, we consider anti-coproximinal subspaces of the space of all scalar valued bounded (continuous) functions. Unlike the  geometric conditions obtained for anti-coproximinal and strongly anti-coproximinal subspaces in the previous section, we conduct an analytic study of such subspaces in the space $\ell_\infty(K)$ and $C(K).$ We first characterize the anti-coproximinal  and strongly anti-coproximinal subspaces in $\ell_\infty(K).$

\begin{theorem}\label{anticoproximinal l_infty}
	Let $K$ be a nonempty set and  let $\mathbb{Y}$ be a proper closed subspace of $\ell_\infty(K).$ Then the following are equivalent: 
		\begin{itemize}		
		\item[(i)] $\mathbb{Y}$ is strongly anti-coproximinal in $\ell_\infty(K)$.
			\item[(ii)] $\mathbb{Y}$ is anti-coproximinal in $\ell_\infty(K)$.
			\item[(iii)] for any $k \in K,$  there exists $f \in \mathbb{Y}$ such that $|f(k)|> \lim_{n \to \infty} |f(k_n)|, $ $ ~  \forall \{k_n\} \subset K$ with $k_n \neq k$ for all but finitely many $n \in \mathbb{N}.$ 
	\end{itemize} 
\end{theorem}

\begin{proof}
	We begin the proof by noting that (i) $\implies$ (ii) holds trivially. Next we prove (ii) $\implies$ (iii). 
	 Suppose on the contrary that there exists $k_0 \in K$ such that for any $f \in \mathbb{Y}, \lim |f(k_n)| \geq |f(k_0)|,$ for some $ \{k_n\} \subset K$ satisfying   $k_n \neq k_0,$ for all but finitely many $n \in \mathbb{N}.$
	 Define $g: K \to \mathbb{K}$ such that 
	 $$g(k_0)=1 ~ \text{and}~ g(k)=0, ~ \forall k \in K \setminus\{k_0\}.$$
	  Clearly $\|g\|=1.$ Let $f \in \mathbb{Y}$ and consider $\{k_n\} \subset K$ such that $k_n \neq k_0,$ for all but finitely many $n \in \mathbb{N}$  and $\lim |f(k_n)|=\|f\|.$ Since $ g(k_n)=0,$ for all but finitely many $n \in \mathbb{N},$ using Theorem \ref{approximate:l(K)},  we obtain $f \perp_B g.$ Therefore, $\mathbb{Y} \perp_B g.$ In other words, $0$ is the best coapproximation to $g$ out of $\mathbb{Y},$ which is clearly a contradiction.

	Let us now prove (iii) $\implies$ (i), to finish the proof of the theorem. Suppose on the contrary that $\mathbb{Y}$ is not strongly  anti-coproximinal in $\ell_\infty(K)$. Then there exists a non-zero $h \in S_{\ell_\infty(K)}$ and $\epsilon \in [0,1)$ such that  $\mathbb{Y} \perp_B^\epsilon h.$  For each $k \in K,$ take $\widetilde{f_k} \in S_\mathbb{Y} $ such that $1=|\widetilde{f_k}(k)|> \lim |\widetilde{f_k}(k_n)|, $    $  ~\forall \{k_n\} \subset K$ satisfying $k_n \neq k,$ for all but finitely many $n \in \mathbb{N}.$   As $\widetilde{f_k}\perp_B^\epsilon h$, using Theorem \ref{approximate:l(K)}, 
		\begin{equation}\label{eqn1}
				co(\{\lim \gamma_n h(k_n): k_n \in K, \gamma_n \in \mathbb{K}, |\gamma_n|=1~  \forall n \in \mathbb{N}, \lim \gamma_n \widetilde{f_k}(k_n)=1\}) \cap \mathcal{D}(\epsilon) \neq \emptyset.
		\end{equation}
	As $1=|\widetilde{f_k}(k)|> \lim |\widetilde{f_k}(k_n)|, $    $  \forall \{k_n\} \subset K$ satisfying $k_n \neq k,$ for all but finitely many $n \in \mathbb{N},$ it is immediate that (\ref{eqn1}) is equivalent to 
		\begin{equation} \label{eq2}
		co(\{\lim \gamma_n h(k):  \gamma_n \in \mathbb{K}, |\gamma_n|=1 ~  \forall n \in \mathbb{N}, \lim \gamma_n \widetilde{f_k}(k)=1\}) \cap \mathcal{D}(\epsilon) \neq \emptyset.
	\end{equation}
 It is straightforward to see that $$\{\lim \gamma_n h(k):  \gamma_n \in \mathbb{K}, |\gamma_n|=1 ~ \forall n \in \mathbb{N}, \lim \gamma_n \widetilde{f_k}(k)=1\} = \{\overline{\widetilde{f_k}(k)} h(k) \}.$$
	This implies that $\overline{\widetilde{f_k}(k)} h(k) \in \mathcal{D}(\epsilon) \implies | h(k)| \leq \epsilon.$
	Observe that 
	\[
1=	\|h\|= \sup_{k \in K}|h(k)| \leq \epsilon < 1,
	\]
	a contradiction.

\end{proof}

We next show that $C_0(K)$ is strongly anti-coproximinal in $\ell_\infty(K),$ whenever $K$ is perfectly normal. Recall that a topological space $K$ is perfectly normal if $K$ is normal and every closed set of $K$ is a $G_{\delta}$ set.

\begin{cor}
	Let $K$ be locally compact perfectly normal space. Then $C_0(K)$ is strongly anti-coproximinal in $\ell_\infty(K).$
\end{cor}

\begin{proof}
	Let $k \in K$ and let $U$ be an open set containing $k.$  Using the Uryshon's Lemma, there exists a continuous function $f: K \to [0,1]$ such that $f^{-1} 
	(1)= \{k\}$ and $f^{-1} (0)= K \setminus U.$ Clearly, $f \in C_0(K)$ and $f$ satisfies the condition (iii) of Theorem \ref{anticoproximinal l_infty}. This finishes the proof. 
\end{proof}

The following result is an easy consequence of  Theorem  \ref{anticoproximinal l_infty}.

\begin{cor}\label{anticoproximinal:l_infty}
	Let $\mathbb{Y}$ be a subspace of  $\ell_\infty.$ Then the following are equivalent: 
	\begin{itemize}
		\item[(i)] $\mathbb{Y}$ is strongly anti-coproximinal in $\ell_\infty$.
		\item[(ii)] $\mathbb{Y}$ is anti-coproximinal in $\ell_\infty$.
		\item[(iii)] for each $r \in \mathbb{N},$  there exists $\widetilde{y}=(y_1, y_2, \ldots ) \in \mathbb{Y}$ such that
		\subitem$(a)$  $|y_r| > |y_n|~ \forall n \in \mathbb{N} \setminus \{r\},$
		\subitem $(b)$ $|y_r| > \lim |y_{n_k}|,$  for any sequence  $\{n_k\}_{k \in \mathbb{N}} $ satisfying $n_k \neq r$ for all but finitely many $k \in \mathbb{N}$.

	\end{itemize} 
\end{cor}

It is clear from Corollary \ref{anticoproximinal:l_infty} that	$c_0$ is a strongly anti-coproximinal subspace of $\ell_\infty.$

We next characterize anti-coproximinal and strongly anti-coproximinal subspaces in $c_0$ and $c.$

\begin{prop}\label{anticoproximinal:c_0}
	Let $\mathbb{X}= c_0$ or $c$ and let $\mathbb{Y}$ be a subspace of $\mathbb{X}.$ Then the following are equivalent: 
	\begin{itemize}
		\item[(i)] $\mathbb{Y}$ is strongly anti-coproximinal in $\mathbb{X}$.
		\item[(ii)] $\mathbb{Y}$ is anti-coproximinal in $\mathbb{X}$.
		\item[(iii)] for any $r \in \mathbb{N},$  there exists $\widetilde{y}=(y_1, y_2, \ldots ) \in \mathbb{Y}$ such that $|y_r| > |y_n|, ~ \forall n \in \mathbb{N} \setminus \{r\}.$ 
	\end{itemize} 
\end{prop}

\begin{proof}
	Clearly, (i) $\implies$ (ii) holds trivially. We prove that (ii) $\implies$ (iii). Suppose on the contrary that there exists $r_0 \in \mathbb{N}$ such that for any $\widetilde{y} =(y_1, y_2, \ldots ) \in \mathbb{Y},$  $|y_{r_0} |\leq |y_s|,$ for some $ s\neq r_0.$ Take $x= ( x_1, x_2, \ldots) $ such that $x_{r_0}=1$ and $x_n=0 \, \, \forall n \in \mathbb{N}\setminus \{r_0\}.$ Let $y=(y_1, y_2, \ldots) \in \mathbb{Y}.$ If $\lim |y_n| \neq \|y\|,$ then there exists $k \in \mathbb{N}$ such that $k \neq r_0$ and $|y_k|= \|y\|.$  Since $x_k=0$ for all $k \neq r_0$, we get $0 \in co(\{ \overline{y_n}x_n: |y_n|=\|y\|\}),$ therefore following \cite[Th. 2.12]{BRS}, $y \perp_B x.$
	Now consider the case that  $\lim|y_n|= \|y\|.$ As $\lim |x_n|=0,$ we have $0 \in co(\{ \overline{y_n}x_n: |y_n|=\|y\|\} \cup \{\lim \overline{y_n}x_n\}).$ Again following \cite[Th. 2.12]{BRS}, $ y \perp_B x.$ This implies that $\mathbb{Y} \perp_B x.$ This contradicts that $\mathbb{Y}$ is anti-coproximinal in $\mathbb{X}$.\\
	
	We next prove that (iii) $\implies$ (i). Suppose on the contrary that there exists an $\epsilon \in [0,1)$ and a nonzero $x = (x_1, x_2, \ldots   ) \in \mathbb{X}$ such that $\mathbb{Y} \perp_B^\epsilon x.$   Let $n \in \mathbb{N}.$ There exists $\widetilde{y_n}=(y_1, y_2, \ldots) \in \mathbb{Y}$ such that $|y_n| > |y_i|, ~\forall i \in \mathbb{N} \setminus \{n\}.$ Let $e_n^* \in \mathbb{X}^*$ be such that $e_n^*(u_1, u_2, \ldots)=u_n, ~\forall (u_1, u_2, \ldots) \in  \mathbb{X}.$ It is clear that $J(\widetilde{y_n})=\{e_n^*\}.$ Since $\widetilde{y_n} \perp_B^\epsilon x,$ applying Lemma \ref{Chem}, we get $| e_n^*(x)| \leq \epsilon \|x\|.$ This implies $|x_n|\leq \epsilon\|x\|.$ Observe that 
	\[
	\|x\|= \sup_{n \in \mathbb{N}} |x_n| \leq \epsilon \|x\| < \|x\|.
	\]
	This is a contradiction. Hence the theorem.
\end{proof}

	

Our next result shows that $\ell_\infty (\mathbb{R})$ does not admit a finite-dimensional polyhedral subspace which is anti-coproximinal.

\begin{prop}\label{no:anticoproximinal}
	Let $\mathbb{X}= \ell_\infty(\mathbb{R}).$ Any finite-dimensional polyhedral subspace of $\mathbb{X}$ is not anti-coprxominal.
\end{prop}

\begin{proof}
	Let $\mathbb{Y}$ be a finite-dimensional polyhedral subspace of $ \ell_\infty(\mathbb{R}). $ Suppose on the contrary that $\mathbb{Y}$ is anti-coproximinal in $\mathbb{X}$. From  Corollary \ref{anticoproximinal:l_infty}, we see that for any $n \in \mathbb{N},$ there exists $\widetilde{y_n} =(y_1, y_2, \ldots)\in \mathbb{Y}$ such that $|y_n| > |y_i| ~ \forall i \in \mathbb{N} \setminus \{n\}.$  For each $n \in \mathbb{N},$ define $$F_n= \{ x=(x_1, x_2, \ldots) \in \mathbb{X}: x_n=1,  |x_i| \leq 1 ~ \forall i \in \mathbb{N} \setminus\{ n \} \}.$$
	Clearly, $F_n \subset S_{\mathbb{X}}.$
	We claim that $F_n$ is a face of $B_{\mathbb{X}}$.  Let $\widetilde{u} =(u_1, u_2, \ldots), \widetilde{w}=(w_1, w_2, \ldots) \in S_{\mathbb{X}}$ be such that $(1-t) \widetilde{u} + t \widetilde{w} \in F_n.$ Clearly, $|u_i|, |w_i| \leq 1.$ So, $$(1-t) u_n+ tw_n =1 \implies u_n = w_n =1.$$ Therefore, $\widetilde{u}, \widetilde{w} \in F_n.$ This implies $F_n$ is a face  of $B_{\mathbb{X}}$. Let $F_n'= F_n \cap B_{\mathbb{Y}}.$ Clearly, $\widetilde{y_n} \in F_n'$ and  $F_n'$ is a face of $B_{\mathbb{Y}}.$ This implies that $B_{\mathbb{Y}}$ contains infinitely many faces, which contradicts the fact that $\mathbb{Y}$ is a polyhedral Banach space.  This completes the proof.
	

\end{proof}

\begin{remark}
	Similarly, using Proposition \ref{anticoproximinal:c_0}, we can show that the space $c_0$ and $c$ do not contain a finite-dimensional polyhedral subspace which is anti-coproximinal.
\end{remark}

It is well known that every finite-dimensional subspace of $c_0$ is polyhedral.  Thus the following corollary is immediate.

\begin{cor}\label{c_0}
	Any finite-dimensional subspace of $c_0$ is not anti-coprxominal.
\end{cor}

Unlike $c_0,$ we give an example of a finite-dimensional strongly anti-coproximinal subspace of $c$ and $\ell_\infty.$

\begin{example}
	Let $\mathbb{X}= c$ or $\ell_\infty.$ Suppose 
	\begin{eqnarray*}
		\widetilde{u}&=& (\cos \frac{\pi}{2}, \cos \frac{\pi}{4}, \ldots, \cos \frac{\pi}{2n}, \ldots ),\\ \widetilde{v}&=& (\sin \frac{\pi}{2}, \sin \frac{\pi}{4}, \ldots, \sin \frac{\pi}{2n}, \ldots ).
	\end{eqnarray*}	
	Let $\mathbb{Y}= span \{\widetilde{u}, \widetilde{v}\}.$ For any $n \in \mathbb{N},$ take $\widetilde{y}= \cos \frac{\pi}{2n} \widetilde{u} + \sin \frac{\pi}{2n} \widetilde{v}.$ It is easy to observe that $|y_n| > |y_i|~ \forall i \in \mathbb{N} \setminus \{n\},$ where $\widetilde{y}=(y_1, y_2, \ldots).$ So, using Proposition \ref{anticoproximinal:c_0} and Corollary \ref{anticoproximinal:l_infty}, we conclude that $\mathbb{Y}$ is strongly anti-coproximinal in $\mathbb{X}$, whereas from Proposition \ref{no:anticoproximinal}, it follows that $\mathbb{Y}$ is not polyhedral.
\end{example}

Though $c_0$ does not have any finite-dimensional strongly anti-coproximinal subspace (see Corollary \ref{c_0}), in the following example we consider an infinite-dimensional subspace which is strongly anti-coproximinal in $c_0$.

\begin{example}
	Let $\mathbb{Y}=\{ (x_1, x_2, \ldots) \in c_0: x_1+ x_2 + x_3=0\}.$ Let 
	\begin{eqnarray*}
		\widetilde{y_1}&=& (1, -\frac{1}{2}, - \frac{1}{2}, 0, \ldots, 0)\\
		\widetilde{y_2}&=& ( -\frac{1}{2}, 1, - \frac{1}{2}, 0, \ldots, 0)\\
		\widetilde{y_3}&=& ( -\frac{1}{2}, - \frac{1}{2}, 1,0, \ldots, 0)\\
		\widetilde{y_i}&=& (0, 0, \ldots, 0 \underset{i-th}{, 1,} 0, \ldots), \quad \forall i > 3,
	\end{eqnarray*}
Clearly, $\widetilde{y_n} \in \mathbb{Y},$ for each $n \in \mathbb{N}.$
	It is immediate that for any $n \in \mathbb{N},$ $\widetilde{y_n}$ satisfies the sufficient condition of Proposition    \ref{anticoproximinal:c_0}. So, $\mathbb{Y}$ is strongly anti-coproximinal in $c_0$.
\end{example}

In \cite[Cor. 3.2]{PS}, Papini et al. gave a sufficient condition for anti-coproximinal subspaces in $C(K).$ In the following theorem we provide a necessary condition for the same. 

\begin{prop}\label{necessary}
	Let $\mathbb{Y}$ be an anti-coproximinal subspace of $C_0(K),$ where $K$ is a locally compact normal space. Then  for each non-empty open subset $U$ of $K,$ there exists an $f \in \mathbb{Y}$ such that $M_f \subset U.$
\end{prop} 

\begin{proof}
	Suppose on the contrary that there exists a non-empty open set $U \subset K$ such that for any $f \in \mathbb{Y},$ $M_f \cap ( K \setminus U) \neq \emptyset.$ In other words, for any $f \in \mathbb{Y}$ there exists $k_f \in K \setminus U$ such that $|f(k_f)| \geq |f(k)|,~ \forall k \in K \setminus \{k_f\}.$ Let $k_0 \in U.$
	 Using  Lemma \ref{Uryshon} there exists a continuous function $g: K \to [0,1]$ such that $g(k_0)=1$ and $g(k)=0,~ \forall k \in K \setminus U.$ Clearly, $g \in C_0(K).$ Let $f \in \mathbb{Y}.$ Since $\|f\|= |f(k_f)|,$ we infer that $\delta_{k_f} \in J(f),$  where $\delta_{k_f}: C_0(K) \to \mathbb{K}$ is given by $\delta_{k_f}(h)=h(k_f), ~\forall h \in C_0(K).$ As $k_f \in K \setminus U,$ we note that $\delta_{k_f}(g)=g(k_f)=0.$ So, using Lemma \ref{James}, we get $f \perp_B g.$ Hence $\mathbb{Y} \perp_B g.$ This contradicts the fact that $\mathbb{Y}$ is anti-coproximinal in $C_0(K)$.

\end{proof}

An immediate corollary of the above theorem is given below.	
\begin{cor}
	Let $K$ be a locally compact normal space and let $\mathbb{Y}$ be a proper closed subspace of $C_0(K).$ Suppose that there exists an element $k_0 \in K$ such that  for any $f \in \mathbb{Y},$ $f(k_0)=0.$ Then $\mathbb{Y}$ is not anti-coproximinal in $C_0(K).$
\end{cor}

We next provide a sufficient condition for strongly anti-coproximinal subspaces.

\begin{prop}
	Let $\mathbb{Y}$ be a proper closed subspace of $C_0(K),$ where $K$ is a locally compact Hausdorff space. Suppose that  $D \subset K$ is dense in $K.$ 
If	for each $k \in D,$ there exists an $f \in \mathbb{Y}$ such that $|f(k)|> |f(k')|,$ for all $k' \in K \setminus \{k\},$ then $\mathbb{Y}$ is strongly anti-coproximinal in $C_0(K).$
\end{prop}

\begin{proof}
	Suppose on the contrary that $\mathbb{Y}$ is not strongly anti-coproximinal in $C_0(K).$ Then there exists $\epsilon \in [0,1), g \in C_0(K) \setminus \mathbb{Y}$  such that $\mathbb{Y} \perp_B^{\epsilon} g.$ Let $k \in D$ and let $\widetilde{f_k} \in \mathbb{Y}$ such that $|\widetilde{f_k}(k)| > |\widetilde{f_k}(k')|,~\forall k' \in K \setminus \{k\}.$ So, $\|\widetilde{f_k}\|= |\widetilde{f_k}(k)|.$ It is easy to observe that $J(\widetilde{f_k})=\{\delta_k\},$ where $\delta_{k}: C_0(K) \to \mathbb{K}$ is given by $\delta_{k}(h)=h(k), ~\forall h \in C_0(K).$ As $\widetilde{f_k} \perp_B^\epsilon g,$ using Lemma \ref{Chem}, we get
	$$|\delta_k(g)| \leq \epsilon \|g\| \implies |g(k)| \leq  \epsilon \|g\|.$$ Observe that 
	\[
	\|g\|= \sup_{k \in K} |g(k)|= \sup_{k \in D} |g(k)|\leq \epsilon \|g\| < \|g\|,
	\]	
	a contradiction. Thus, $\mathbb{Y}$ is not strongly anti-coproximinal.
\end{proof}

In addition, if we assume that $K$ is locally connected, then we can prove that the above necessary condition given in Proposition \ref{necessary} is also sufficient for anti-coproximinal as well as for strongly anti-coproximinal subspaces. From now on  we consider the notation $C_0(K)$ for the space $C_0(K, \mathbb{R}).$ We require the following classical result for our purpose.

\begin{lemma}\label{singer}\cite[Th. 2.4]{CSW}
	Let $\mathbb{X}$ be a real Banach space. Suppose  $x,y \in \mathbb{X}$ and $\epsilon \in [0,1).$ Then $x \perp_B^\epsilon y$ if and only if there exists $\phi, \psi \in Ext(B_{\mathbb{X}^*}) \cap J(x)$  such that $|((1-t) \phi + t \psi)y|\leq \epsilon \|y\|,$ for some $0 \leq t \leq 1.$
\end{lemma}

	\begin{theorem}\label{C(K)}
		Let $K$ be a locally connected and locally compact  normal space. Suppose that  $\mathbb{Y}$ is a proper closed subspace of $C_0(K)$. Then the following are equivalent:
		\begin{itemize}
			\item[(i)] $\mathbb{Y}$ is strongly anti-coproximinal in $C_0(K)$.
			\item[(ii)] $\mathbb{Y}$ is anti-coproximinal  in $C_0(K)$.
			\item[(iii)] For each nonempty open subset $U$ of $K,$ there exists an $f \in \mathbb{Y}$ such that $M_f \subset U.$
			
		\end{itemize} 
	\end{theorem}

	\begin{proof}
		We begin this proof by noting that (i) $\implies$ (ii) holds trivially. Also, (ii) $\implies$ (iii) follows from  Proposition \ref{necessary}.\\
		Therefore, we only prove that (iii) $\implies $ (i).
		 Suppose on the contrary that there exists $\epsilon \in [0,1), g \in C_0(K) \setminus \mathbb{Y}$  such that $\mathbb{Y} \perp_B^{\epsilon} g.$ As $M_g$ is nonempty, suppose $|g(k')|=\|g\|,$ for some $k'\in K.$ As $g$ is continuous, there exists an open set $W$ containing $k'$  such that $|g(w)|> \epsilon\|g\|,$ for all $w \in W.$ Since $K$ is locally connected, it follows that there exists an open connected set $V$ such that $k'\in V \subset W.$ Then clearly, $|g(v)|> \epsilon\|g\|,$ for all $v \in V.$   Now, from (iii), there exists an $f \in \mathbb{Y}$ such that $M_{f} \subset V.$ Since $f \perp_B^{\epsilon} g,$ using Lemma \ref{singer}, there exist $\phi, \psi \in Ext(B_{C_0(K)^*}) \cap J(f)$ such that
		  $$|((1-t)\phi + t \psi)(g)|\leq \epsilon \|g\|,$$ 
		  for some $t \in [0, 1].$ It is well known that (cf. \cite[Chapter I, 1.10]{singer}) $\phi = \delta_{k_1}$ and $\psi=\delta_{k_2}$ for some $k_1, k_2 \in M_f,$ where $\delta_{k_i}: C_0(K) \to \mathbb{R}$ such that $\delta_{k_i}(h)= h(k_i), ~ \forall h \in C_0(K).$ Therefore,    $|((1-t)\delta_{k_1}+ t\delta_{k_2})(g)| \leq \epsilon\|g\|. $ So, $|(1-t) g(k_1)+ t g(k_2)| \leq \epsilon \|g\|.$
		   As $k_1, k_2 \in V,$ $V$ is connected and $g$ is   continuous, it is straightforward to see that there exists $k_0 \in V$ such that  $ |g(k_0)| \leq  \epsilon \|g\|,$ a contradiction. So $\mathbb{Y}$ is strongly anti-coproximinal  in $C_0(K)$, as desired.
	\end{proof}

\begin{remark}
  	It is well known that $\mathbb{X}^*$ is embedded into  $C(S_{\mathbb{X}})$  via the map $\Psi: z^* \longrightarrow z^*|_{S_\mathbb{X}}.$ Whenever $\mathbb{X}$ is finite-dimensional, we consider the subspace $\Psi(\mathbb{X}^*)$ in $C(S_{\mathbb{X}}).$ Note that for any $z^* \in \mathbb{X}^*,$ $M_{z^*}$ always contains the antipodal points of $S_{\mathbb{X}}$. Therefore, applying the condition (iii) of Theorem \ref{C(K)}, we can observe that $\Psi(\mathbb{X}^*)$ is not anti-coproximinal in $C(S_\mathbb{X}).$ 
\end{remark}



Next we show that any finite-codimensional subspace of $C_0(K)$ is strongly anti-coproximinal, whenever $ K $ has additional nice properties. 
  
\begin{theorem}\label{finite-codimension}
	Let $K$ be a locally connected, locally compact and perfectly normal space such that $K$ does not contain any isolated point. Then any finite-codimensional subspace  of $C_0(K)$ is strongly anti-coproximinal.
\end{theorem}

\begin{proof}
	Let $\mathbb Y$ be an $m$-codimensional subspace  of $C_0(K).$ Suppose on the contrary that $\mathbb Y$ is not strongly anti-coproximinal in $C_0(K)$. Applying Theorem  \ref{C(K)}, there exists an nonempty open set $U \subset K$ such that for any $f \in \mathbb Y, M_f \cap ( K \setminus U)\neq \emptyset.$ So, for any $f \in \mathbb{Y},$ take $k_f \in K \setminus U$ such that $\|f\|= |f(k_f)|.$
	 Let $k_1, k_2, \ldots, k_{m+1} \in U.$ As $K$ is Hausdorff, let $V_1, V_2, \ldots, V_m, V_{m+1} \subset U$ be the open sets of $K$ such that $ k_i \in  V_i,~V_i \cap V_j= \emptyset,$  $\forall i,j \in \{1, 2, \ldots, m+1\}, i \neq j.$ As $K$ is perfectly normal, using  Lemma \ref{Uryshon},  we assert that  there exist  continuous functions $g_i: K \to [0,1]$ such that $$g_i^{-1}(\{1\}) = \{k_i\} ~ \text{and}~g_i^{-1}(\{0\})= K \setminus V_i, \, \forall \,i = 1,2, \ldots, m+1.$$
	Clearly, for each $1 \leq i \leq m+1,$ $M_{g_i}=\{k_i\}$ and so, $g_i \in C_0(K) \setminus \mathbb Y.$ Moreover it is obvious that $g_i(k_j)=0,$ whenever $i \neq j.$ Observe that $\{g_1, g_2, \ldots, g_{m+1}\}$ is a linearly independent set. Since $ \mathbb Y $ is an $m$-codimensional subspace  of $C_0(K)$, so 
	                               $$g_{m+1} = \alpha_1 g_1 + \alpha_2 g_2 + \ldots+ \alpha_m g_m+  f,  \mbox{where}\quad \alpha_i \in \mathbb{K},  f \in \mathbb Y.$$ 
	Since $k_f \in K \setminus U$ such that $\|f\|= |f(k_f)|,$ we have
	                                    $$g_{m+1}(k_f)=  \sum_{i=1}^{m}\alpha_i g_{i}(k_f) + f(k_f) \implies f(k_f)=0.$$ 
	 As $\|f\|= |f(k_f)|,$ this implies $f=0,$ therefore, $g_{m+1}= \sum_{i=1}^{m}\alpha_i g_i,$ a contradiction. This establishes the theorem.
\end{proof}

\begin{remark}
	
	While studying the specialty of inner product spaces among Banach spaces, James stated an interesting result in \cite[p. 564]{J2}:	``\emph{For no hyperspace $H$ of the space $C[a,b]$ of continuous functions defined on $[a,b] \subset \mathbb{R}$ there is an element $f \in C[a,b]$ with $H \perp_B f.$ }''
	Observe that this is the least favorable case for the space to be an inner product space from the perspective of  \cite[Th. 4]{J2}, which states that a Banach space $\mathbb X$ is an inner product space if and only if for  a given hyperspace $H$ there exists an element $ z \in \mathbb X$ such that $ H \perp_B z.$  To the best of our knowledge,  $C[a,b]$ is the only known example of a Banach space with such a property. Following Theorem \ref{finite-codimension}, we can strengthen the statement above given by James in the following way:\\
	
	\noindent	\emph{Whenever $K$ is locally connected, compact and perfectly normal space, there is no  subspace $\mathbb{Y}$ of $C(K)$ with finite-codimension   such that given any $\epsilon \in [0, 1),$ there is an element $f \in C(K)$ satisfying $\mathbb{Y} \perp_B^\epsilon f.$}

\end{remark}
 


Our final goal in this section is to investigate whether $C(K, \mathbb{Y})$ is  anti-coproximinal (strongly anti-coproximinal) in $C(K , \mathbb{X})$ if $\mathbb{Y}$ is anti-coproximinal (strongly anti-coproximinal) in $\mathbb{X}$ and vice versa. We require the following lemma for our purpose.

\begin{lemma}\label{lemma:outside}
 	Let $K$ be a  compact perfectly normal space and let $\mathbb{Y}$ be a subspace of  $\mathbb{X}.$ Suppose  $f \in S_{C(K, \mathbb{X})}$ is such that $C(K,\mathbb{Y}) \perp_B^\epsilon f,$ where $\epsilon \in [0,1).$ Then $f(M_f) \cap (\mathbb{X} \setminus \mathbb{Y}) \neq \emptyset.$
\end{lemma}

\begin{proof}
	Suppose on the contrary that $f(M_f) \subset \mathbb{Y}.$ Let $k_0 \in M_f$ and let $U \subset K$ be an open set containing $k_0.$ As $K$ is perfectly normal, there exists a  continuous function $\phi: K \to [0,1]$ such that $\phi^{-1}(\{1\})= \{k_0\}$ and $\phi^{-1}(\{0\})= K \setminus U.$ Consider the function $g: K \to \mathbb{X}$ such that $g(k)= \phi(k) f(k_0) ~ \forall k \in K.$ Clearly, $g$ is continuous. As $f(k_0) \in \mathbb{Y},$ it follows trivially that $g(k) \in \mathbb{Y}~ \forall k \in K.$ So, $g \in C(K, \mathbb{Y}).$ It is easy to see that  $M_g= \{k_0\}.$
	Consider the set 
	\[
	S= 	co\bigg(\{y^*(f(k)): k \in K, y^* \in C, y^*(g(k))= \|g\|\} \bigg)  ,
	\]
		where $B_{\mathbb{X}^*}= \overline{co(C)}^{w^*}.$ As $M_g=\{k_0\},$ it is easy to observe that $y^*(g(k))= \|g\| $ implies that $k=k_0.$ So, 
		\[
		S= 	co(\{y^*(f(k_0)):  y^* \in C, y^*(g(k_0))= 1\}) .
		\]
	Observe that $f(k_0)= g(k_0)$, so $y^*(f(k_0))=1,$ where $y^*(g(k_0))=1.$ Therefore $S = \{1\}.$ Now following Theorem \ref{approximate:continuous},  $g \perp_B^\epsilon f$ implies that  $S \cap \mathcal{D}(\epsilon) \neq \emptyset,$ a contradiction. This completes the proof.
	
\end{proof}

\begin{theorem}\label{C(K, X)}
	Let $K$ be a  compact perfectly normal space and let $\mathbb{Y}$ be a subspace of  $\mathbb{X}.$ Then 
	\begin{itemize}
	
		\item[	(i)] $C(K, \mathbb{Y})$ is strongly anti-coproximinal in $C(K, \mathbb{X})$ if and only if $\mathbb{Y}$ is strongly anti-coproximinal in $\mathbb{X}.$
		
			\item[	(ii)] $C(K, \mathbb{Y})$ is anti-coproximinal in $C (K, \mathbb{X})$ if and only if $\mathbb{Y}$ is anti-coproximinal in $\mathbb{X}.$
	\end{itemize} 
\end{theorem}

\begin{proof}
	
	(i)  We first prove the sufficient part. Suppose on the contrary that there exists a nonzero $h \in C(K, \mathbb{X})$ and $\epsilon \in [0,1)$ such that $C(K, \mathbb{Y}) \perp_B^\epsilon h.$ Let $\|h\|=1.$
	As $h \neq 0, $ using Lemma \ref{lemma:outside}  there exists $k_0 \in K$ such that $h(k_0) \in  S_\mathbb{X} \setminus S_\mathbb{Y}.$
	 Let $U \subset K$ be an open set containing $k_0.$ As $K $ is perfectly normal, using Lemma \ref{Uryshon}, there exists a continuous function $\phi: K \to [0,1] $	such that $\phi^{-1}(\{1\})=\{k_0\}$ and $\phi^{-1}(\{0\})= K \setminus U.$ 
	Let $y \in \mathbb{Y}.$ Define $f_y : K \to \mathbb{Y}$ such that $f_y(k)= \phi(k)y~ \forall k \in K.$ Clearly, $f_y \in C(K, \mathbb{Y})$ and $f_y(k)=0, ~ \forall k \in K \setminus U.$ Moreover, $f_y(k_0)= y$ and for any $k \in K \setminus \{k_0\}, $ $\|f_y(k)\|= \|\alpha(k)y\| < \|y\|.$ So, $M_{f_y}=\{k_0\}.$ As  $f_y \perp_B^\epsilon h,$ using Theorem \ref{approximate:continuous}, 
	\[
	co(\{ y^*(h(k_0)): y^* \in C, y^*(f_y(k_0))=\|f_y\|\}) \cap \mathcal{D}(\epsilon)\neq \emptyset.
	\]
	Applying Carath\'{e}odory’s Theorem  (see \cite[Th. 17.1]{R}), there exists real scalars $t_i(i=1,2,3), t_i > 0, \sum_{i=1}^{3} t_i=1$ such that 
	\[
   \bigg|	\sum_{i=1}^{3} t_i y_i^*(h(k_0))\bigg| \leq \epsilon,
	\]
	where $y_i^*(f_y(k_0))=\|f_y\|=\|f_y(k_0)\|.$ For each $1 \leq i \leq 3,$ $y_i^* \in J(f_y(k_0)),$ which implies that  $\sum_{i=1}^{3} t_i y_i^* \in J(f_y(k_0)).$ Since $\|h(k_0)\|=1,$ it follows that $|	\sum_{i=1}^{3} t_i y_i^*(h(k_0))| \leq \epsilon\|h(k_0)\|$ and therefore, following Lemma \ref{Chem}, we get that $f_y(k_0) \perp_B^\epsilon h(k_0)$.
	Since $f_y(k_0)= y,$ we have 
	  $  y \perp_B^\epsilon h(k_0).$
   This implies $\mathbb{Y} \perp_B^\epsilon h(k_0),$ which contradicts that $\mathbb{Y}$ is strongly anti-coproximinal in $\mathbb{X}$.

		Let us now prove the necessary part. Suppose on the contrary that $\mathbb{Y}$ is not strongly anti-coprximinal in $\mathbb{X}.$ Then there exists an  $x \in \mathbb{X} \setminus \mathbb{Y}$ and $\epsilon \in [0,1)$ such that $\mathbb{Y} \perp_B^\epsilon x.$ Define $g : K \to \mathbb{X}$ such that $g(k)= x ~\forall k \in K.$ Let $f \in C(K, \mathbb{Y}),$ and let $k \in M_f.$ As $f(k) \in \mathbb{Y},$ it is clear that $f(k) \perp_B^\epsilon g(k).$ Observe that 
	\[
	\|f + \lambda g\| \geq \|f(k)+ \lambda g(k)\|\geq \|f(k)\|- \epsilon|\lambda|\|g(k)\| \geq  \|f\|- \epsilon|\lambda |\|g\|,
	\]
	which shows that $f \perp_B^\epsilon g.$ However, this implies that $C(K, \mathbb{Y}) \perp_B^\epsilon g,$ contradicting that $C(K, \mathbb{Y})$ is strongly anti-coproximinal in $C(K, \mathbb{X}).$ This completes the proof of (i).
	
	(ii) Let us first prove the sufficient part. Suppose on the contrary that there exists a nonzero $h \in C(K, \mathbb{X}) $ such that  $C(K, \mathbb{Y}) \perp_B h.$ As $h \neq 0, $ then there exists $k_0 \in K$ such that $h(k_0) \in \mathbb{X} \setminus \mathbb{Y}.$  For any $y \in \mathbb{Y},$ let $f_y: K \to \mathbb{Y}$ be the same function defined in (i). Using similar arguments as given in the proof of (i), we obtain that $M_{f_y}=\{k_0\}.$
	 As $f_y \perp_B h,$ applying \cite[Th. 4.3]{MMQRS}, $f_y(k_0) \perp_B h(k_0).$ Since $f_y(k_0)=y,$ we have $y \perp_B h(k_0).$ This implies that $\mathbb{Y} \perp_B h(k_0),$  contradicting that $\mathbb{Y}$ is anti-coproximinal in $\mathbb{X}.$
	
	For the necessary part, suppose on the contrary that $\mathbb{Y}$ is not anti-coproximinal in $\mathbb{X}.$ Then there exists an $x \in \mathbb{X}\setminus \mathbb{Y}$ such that $\mathbb{Y} \perp_B x.$ Let $g : K \to \mathbb{X}$ be the function defined as $g(k)=x~ \forall k \in K.$ Let $f \in C(K, \mathbb{Y})$  and let $k \in M_f.$ As $f(k) \in \mathbb{Y},$  $f(k) \perp_B g(k).$ Observe that 
	\[
	\|f + \lambda g\| \geq \|f(k)+ \lambda g(k)\|\geq \|f(k)\|= \|f\|,
	\]
	proving that $f \perp_B g.$ This implies $C(K, \mathbb{Y}) \perp_B g,$ a contradiction to the fact that $C(K, \mathbb{Y})$ is anti-coproximinal in $C(K, \mathbb{X}).$
	 \end{proof}



\begin{remark}
	We note that the notions of anti-coproximinal and strongly anti-coproximinal subspaces coincide in the  $C(K, \mathbb{R})$ spaces (see Th.  \ref{C(K)}). On the other hand, these two notions do not coincide in $C(K, \mathbb{X}),$ for an arbitrary Banach space  $\mathbb{X}.$  Indeed,  let us take a subspace $\mathbb{Y}$ of $\mathbb{X}$ which is anti-coproximinal  but not strongly anti-coproximinal (see \cite[Example 2.21]{SSGP3}). Then from the above theorem, it follows that $C(K, \mathbb{Y})$ is anti-coproximinal but not strongly anti-coproximinal in $C(K, \mathbb{X})$. 
\end{remark}

\section*{Anti-coproximinality in the space of bounded linear operators}

 In this section, we study the anti-coproximinality  (strongly anti-coproximinality)  in the space of all bounded linear operators defined on Banach spaces using the well known \emph{Bhatia-\v{S}emrl property}.
Recall that an operator $T \in \mathbb{L}(\mathbb{X}, \mathbb{Y})$ satisfies \emph{Bhatia-\v{S}emrl (B\v{S}) Property} \cite{SPH} if for any $A \in \mathbb{L}(\mathbb{X}, \mathbb{Y}),$ $T \perp_B A$ if and only if  $Tx \perp_B Ax,$ for some $x \in M_T.$ Observe that, in general, \emph{Bhatia-\v{S}emrl property} does not hold for an arbitrary bounded linear operator  defined on a Banach space. However the same holds under additional restrictions on the space as well as the operator. For more information on \emph{B\v{S}} property, readers may go through \cite{Kim, SPH, SRSP} and the references therein. We show that if $T \in \mathbb{L}(\mathbb{X}, \mathbb{Y})$ is an  absolutely strongly exposing operator then $T$ satisfies B\v{S} property. Note that  an operator $T \in \mathbb{L}(\mathbb{X}, \mathbb{Y})$ is said to be an absolutely strongly exposing operator (see \cite{Jung}) if there exists $x_0 \in B_{\mathbb{X}}$ such that whenever a sequence $\{x_n\}$ in $B_{\mathbb{X}}$ satisfies $\|Tx_n\| \to \|T\|,$ then there exists a sequence $\{\theta_n\}$ of elements of $\mathbb{T}$ such that $\theta_n x_n \to x_0.$  It is easy to observe that for such an operator $T,$ $M_T = \{\mu x_0: |\mu|=1\},$ for some $x_0 \in S_{\mathbb{X}}.$
The set of absolutely strongly exposing operators of $\mathbb{L}(\mathbb{X}, \mathbb{Y})$ is denoted by $ASE(\mathbb{X}, \mathbb{Y})$.



\begin{theorem}\label{BSP}
	Let $\mathbb{X}$ and $\mathbb{Y}$ be two Banach spaces and let $\epsilon \in [0,1).$ Suppose that  $T \in ASE(\mathbb{X}, \mathbb{Y})$ with $M_T= \{ \mu x_0: |\mu|=1\}.$  Then for any $A \in S_{\mathbb{L}(\mathbb{X}, \mathbb{Y})}$ the following are equivalent:
	\begin{itemize}
		\item[(i)] $T \perp_B^{\epsilon} A$
		\item[(ii)] $ \Omega \cap \mathcal{D}(\epsilon)\neq \emptyset,$ where $\Omega = \{y^*(Ax_0): \, y^*\in J(Tx_0) \}$
	\end{itemize} 
	Moreover, if $M_T \subseteq M_A,$ then $T \perp_B^{\epsilon} A$ if and only if  $Tx_0 \perp_B^\epsilon Ax_0.$
\end{theorem}

\begin{proof}
Following Theorem \ref{approximate:operator}, $T \perp_B^\epsilon A$ if and only if $ co(\Omega') \cap \mathcal{D}(\epsilon) \neq \emptyset,$ where
	\[\Omega' = \bigg\{\lim y_n^*(Ax_n): (x_n, y_n^*) \in S_\mathbb{X} \times S_{\mathbb{Y}^*}, \, \, \lim y_n^*(Tx_n)=\|T\| \bigg\}.\] 
	Clearly, $\Omega  \subseteq \Omega'.$ Since $J(Tx_0)$ is convex, it is easy to see that $\Omega$ is convex. Therefore, to complete the theorem, we only need to show $\Omega' \subseteq \Omega.$
	Suppose that $\lambda \in \Omega'.$ Then there exists $(x_n, y_n^*) \in S_\mathbb{X} \times S_{\mathbb{Y}^*}$ such that $y_n^*(Tx_n)\to \|T\|$ and $y_n^*(Ax_n) \to \lambda.$ Without loss of generality we may assume that $\|Tx_n\| \to c,$ a real number. As  $\|T\|= \lim y_n^*(Tx_n) \leq \lim\|Tx_n\| =c \leq \|T\|,$ it is clear that $\|Tx_n\| \to \|T\|.$ Now, since $T\in ASE(\mathbb{X}, \mathbb{Y})$,  there exists $x \in B_\mathbb{X}$ and $\{\theta_n\} \subset \mathbb{T}$ such that $\theta_n x_n \to x.$ Clearly, $x= \mu x_0,$ for some $\mu \in \mathbb{T}.$ Again, passing onto a subsequence, if necessary, we may assume that $\theta_n \to \theta \in \mathbb{T}.$  Then $x_{n} \to \mu\bar{\theta} x_0$ and so $Tx_{n}\to \mu\bar{\theta} Tx_0.$  Since $B_{\mathbb{Y}^*}$ is weak*-compact, it follows that there exists $y^* \in B_{\mathbb{Y}^*}$ such that $y_n^* \overset{w^*}{\to} y^*.$ 
	Thus  $y_{n}^*(Tx_{n}) \to y^*(\mu\bar{\theta}Tx_0)$. Then $\mu\bar{\theta}y^*(Tx_0)= \lim y_n^*(Tx_n)=\|T\|$ and so $\mu\bar{\theta}y^* \in J(Tx_0).$ Also, $y_n^*(Ax_n) \to \mu\bar{\theta}y^*(Ax_0).$ This implies $\mu\bar{\theta}y^*(Ax_0)= \lim y_n^*(Ax_n)= \lambda.$  Therefore, $\lambda \in \Omega.$

Let $M_T \subseteq M_A.$ So, $\|Ax_0\|= \|Tx_0\|=1.$	 Since $T \perp_B^\epsilon A,$ it follows that (ii) holds. Let $\rho \in \Omega \cap \mathcal{D}(\epsilon).$ Suppose that $\rho= y^*(Ax_0),$ for some $y^* \in J(Tx_0).$ As $|\rho| \leq \epsilon = \epsilon\|Ax_0\|,$  applying Lemma 2.2 we obtain that $Tx_0 \perp_B^{\epsilon} Ax_0.$
\end{proof}

In this context, we note that in \cite[Th. 3.2]{PSM}, considering $\mathbb{X}$ as reflexive and $T, A \in \mathbb{K}(\mathbb{X}, \mathbb{Y})$, it was shown that whenever $M_T \subseteq M_A, T \perp_B^\epsilon A $ if and only if there exists $x \in M_T$ such that $Tx \perp_B^\epsilon Ax.$ 

\begin{theorem}\label{BSP2}
	Let $\mathbb{X}$ and $\mathbb{Y}$ be two Banach spaces. Suppose that  $T \in ASE(\mathbb{X}, \mathbb{Y})$ with $M_T = \{\mu x_0: |\mu|=1\}.$  Then for any $A \in S_{\mathbb{L}(\mathbb{X}, \mathbb{Y})}$, 
 $T \perp_B A$ if and only if 
 $Tx_0 \perp_B Ax_0.$

\end{theorem}

\begin{proof}
	We only prove the necessary part as the sufficient part follows easily. 
	Following Theorem \ref{BSP}, $T \perp_B A$ if and only if $0 \in \Omega,$ where $\Omega= \{y^*(Ax_0): y^*\in J(Tx_0) \}.$ 
	Applying Lemma \ref{James}, we get $Tx_0 \perp_B Ax_0,$ as desired.
	
\end{proof}

 In \cite[Th. 2.4]{Kim}, it was proved that \emph{``For a real Banach space $\mathbb{X}$, if the closed unit  ball $B_{\mathbb{X}}$ is an RNP set, then the set of norm attaining operators satisfying the B\v{S} Property is dense in $\mathbb{L}(\mathbb{X}, \mathbb{Y}),$ for every Banach space $\mathbb{Y}$".} 
In the following result,  we prove this result for both real and complex Banach spaces.

\begin{cor}
	Let $\mathbb{X}$ be a Banach space with Radon-Nikodym Property and let $\mathbb{Y}$ be any Banach space. The set of norm attaining operators satisfying \emph{ B\v S-Property} is dense in $\mathbb{L}(\mathbb{X}, \mathbb{Y}).$
\end{cor}

\begin{proof}
	Since  $\mathbb{X}$ satisfies Radon-Nikodym Property, it follows from  \cite[Th. 5]{B} that for any given Banach space $\mathbb{Y},$ $ASE(\mathbb{X}, \mathbb{Y})$ is dense in $\mathbb{L}(\mathbb{X}, \mathbb{Y}).$  Following Theorem \ref{BSP2}, when $T \in ASE(\mathbb{X}, \mathbb{Y})$, we have that $T$ satisfies \emph{B\v{S} property}, thus completing the proof.
\end{proof}

  We next study the subspace $\mathbb{K}(\mathbb{X}, \mathbb{Y})$ of $ \mathbb{L}(\mathbb{X}, \mathbb{Y}),$ from the perspective of best coapproximation.

\begin{theorem}\label{anti}
	Let $\mathbb{X}$ be a Banach space such that the set of all strongly exposed points of $B_\mathbb{X}$ separates $\mathbb{X}^*.$ Let $\mathbb{Y}$ be any Banach space. Then exactly one  of the following holds true:
	\begin{itemize} 
		\item[(i)] $\mathbb{K}(\mathbb{X}, \mathbb{Y})= \mathbb{L}(\mathbb{X}, \mathbb{Y}).$
		\item[(ii)] $\mathbb{K}(\mathbb{X}, \mathbb{Y})$ is anti-coproximinal in $\mathbb{L}(\mathbb{X}, \mathbb{Y}).$ 
	\end{itemize} 
\end{theorem}

\begin{proof}
	Let us assume that $\mathbb{K}(\mathbb{X}, \mathbb{Y}) \subsetneqq \mathbb{L}(\mathbb{X}, \mathbb{Y}).$ 	Suppose on the contrary that $\mathbb{K}(\mathbb{X}, \mathbb{Y})$ is not anti-coproximinal in $\mathbb{L}(\mathbb{X}, \mathbb{Y})$. Then there exists $S \in \mathbb{L}(\mathbb{X}, \mathbb{Y}) \setminus \mathbb{K}(\mathbb{X}, \mathbb{Y}) $  such that $$\mathbb{K}(\mathbb{X}, \mathbb{Y}) \perp_B S.$$ Without loss of generality, assume that $\|S\|=1.$ Since $s\mhyphen Exp(B_\mathbb{X})$ separates $\mathbb{X}^*$ it is straightforward to see that  there exists $z \in s \mhyphen Exp(B_\mathbb{X})$ such that $Sz \neq 0.$  
	Suppose that $Sz=w$ and $z^* \in J(z)$ such that $M_{z^*}=\{\mu z: |\mu|=1\}.$ Define $T: \mathbb{X}\longrightarrow \mathbb{Y}$ by 
	$$Tx=z^*(x)w, ~ \text{for any}~ x\in \mathbb{X}.$$
	  Note that for any $x \in S_\mathbb{X},$ $\|Tx\| = |z^*(x)|\|w\| \leq \|w\|.$ As $\|Tz\|=|z^*(z)|\|w\|=\|w\|> |z^*(x)|\|w\|=\|Tx\|,$ for all $x\in S_\mathbb{X} \setminus \{\mu z: |\mu|=1\}$, so $\|T\|=\|w\|$ and $M_T=\{\mu z: |\mu|=1\}.$   Clearly, $T \in \mathbb{K}(\mathbb{X}, \mathbb{Y}).$ On the other hand, let $\{x_n\}\subset S_\mathbb{X}$ be such that $\|Tx_n\| \longrightarrow \|T\|=\|w\|.$ Therefore, $\lim \|z^*(x_{n}) w\|= \|T\|=\|w\|.$ This implies $\lim |z^*(x_{n})|=1.$ Therefore, there exists a sequence $\{\theta_n\} \subset \mathbb{T}$ such that $\theta_n z^*(x_n) \longrightarrow 1,$ which implies that $z^*(\theta_nx_n) \longrightarrow 1=z^*(z).$  Since $z$ is strongly exposed point of $B_{\mathbb{X}}$, we have $ \theta_n x_{n} \longrightarrow z.$ Therefore $T \in ASE(\mathbb{X}, \mathbb{Y}).$ Since $\mathbb{K}(\mathbb{X}, \mathbb{Y}) \perp_B S$ we have $T \perp_B S.$ As $M_T= \{\mu z: |\mu|=1\},$ applying Theorem \ref{BSP2}, we obtain that $Tz \perp_B Sz,$ a contradiction. This completes the proof of the theorem.  
\end{proof}

\begin{remark}
It is worth mentioning here that  there does not exist any best coapproximation to the identity operator $I$ out of  the subspace $\mathbb{K}(\mathbb{X}),$ whenever $\mathbb{X}$ is an infinite-dimensional Banach space  and $\mathbb{K}(\mathbb{X})$ is semi-M-ideal  in $span\{I, \mathbb{K}(\mathbb{X})\},$ (see  \cite[Th. 14]{Rao}).
\end{remark}

\begin{theorem}\label{strongly exposed}
	Let $\mathbb{X}, \mathbb{Y}$ be  Banach spaces. Suppose that  $B_{\mathbb{X}}$ is the closed convex hull of its strongly exposed points. Then exactly one  of the following holds true:
	\begin{itemize} 
		\item[(i)] $\mathbb{K}(\mathbb{X}, \mathbb{Y})= \mathbb{L}(\mathbb{X}, \mathbb{Y}).$
		\item[(ii)] $\mathbb{K}(\mathbb{X}, \mathbb{Y})$ is strongly anti-coproximinal in $\mathbb{L}(\mathbb{X}, \mathbb{Y}).$ 
	\end{itemize}
\end{theorem}

\begin{proof}
	Let us assume that $\mathbb{K}(\mathbb{X}, \mathbb{Y}) \subsetneqq \mathbb{L}(\mathbb{X}, \mathbb{Y}).$
	Suppose on the contrary that $\mathbb{K}(\mathbb{X}, \mathbb{Y})$ is not strongly  anti-coproximinal in $\mathbb{L}(\mathbb{X}, \mathbb{Y})$. Then there exists $S \in \mathbb{L}(\mathbb{X}, \mathbb{Y}) \setminus \mathbb{K}(\mathbb{X}, \mathbb{Y}) $ and $\epsilon \in [0,1)$ such that $$\mathbb{K}(\mathbb{X}, \mathbb{Y}) \perp_B^\epsilon S.$$ Without loss of generality we assume $\|S\|=1.$ Since $B_\mathbb{X}$ is the closed convex hull of its strongly exposed points, it follows that there exists $z \in S_\mathbb{X}$ such that $\|Sz\| > \epsilon,$ where $z$ is a strongly exposed point of $B_\mathbb{X}.$ Otherwise, from convexity of norm we can write 
	\begin{eqnarray*}
		\|S\|= \sup \{\|Sx\|: x \in B_{\mathbb{X}}\} &=& \sup \{\|Sx\|: x \in \overline{co(s \mhyphen Exp(B_{\mathbb{X}})) }\} \\ 
		&=& \sup \{\|Sx\|: x \in co(s \mhyphen Exp(B_{\mathbb{X}})) \}\\
		&=& \sup \{\|Sx\|: x \in s \mhyphen Exp(B_{\mathbb{X}})\} \leq \epsilon < 1.
	\end{eqnarray*}
	Suppose that $Sz=w$ and $z^* \in J(z)$ such that $M_{z^*}=\{\mu z: |\mu|=1\}.$ Now we define $T \in \mathbb{L}(\mathbb{X}, \mathbb{Y})$ such that $Tx=z^*(x) w$ for all $x \in \mathbb{X}.$ Now proceeding similarly as in the proof of Theorem \ref{anti}, we obtain that  $T \in ASE(\mathbb{X}, \mathbb{Y}).$ Clearly,  $M_T = \{\mu z: |\mu|=1\}.$
	 Since $\mathbb{K}(\mathbb{X}, \mathbb{Y}) \perp_B^\epsilon S$ we have $T \perp_B^\epsilon S.$ Following Theorem \ref{BSP},  $ \Omega \cap \mathcal{D}(\epsilon)\neq \emptyset,$ where $\Omega = \{ y^*(Sz):  y^*\in J(Tz) \}.$ Since $Tz= Sz =w,$ it follows that  $\Omega= \{\|w\|\}.$ However, $\|w\|= \|Sz\| > \epsilon,$ a contradiction. 
	This establishes the theorem.

\end{proof}

From \cite[p. 121]{DP}, we note that a Banach space $\mathbb{X}$ satisfies the Radon-Nikodym Property if and only if every bounded subset of $\mathbb{X}$ is dentable. Also, in \cite[Th. 9]{Ph} Phelps proved that  every bounded subset of $\mathbb{X}$ is dentable if and only if every bounded closed convex subset of $\mathbb{X}$ is the closed convex hull of its strongly exposed points. Combining these results with Theorem \ref{strongly exposed}, we obtain the following results.

\begin{cor}
	Suppose that $\mathbb{X}$ is a Banach space with the Radon-Nikodym Property. Then either of the following holds true:
	\begin{itemize}
		\item[(i)] $\mathbb{K}(\mathbb{X}, \mathbb{Y})= \mathbb{L}(\mathbb{X}, \mathbb{Y}).$
		\item[(ii)] $\mathbb{K}(\mathbb{X}, \mathbb{Y})$ is strongly anti-coproximinal in $\mathbb{L}(\mathbb{X}, \mathbb{Y}).$ 
	\end{itemize}
\end{cor}

\begin{cor}
 Let $\mathbb{X}$ be a Banach space and let $\mathbb{Y}$ be a reflexive Banach space. Then either of the following holds true:
 \begin{itemize}
 	\item[(i)] $\mathbb{K}(\mathbb{X}, \mathbb{Y})= \mathbb{L}(\mathbb{X}, \mathbb{Y}).$
 	\item[(ii)] $\mathbb{K}(\mathbb{X}, \mathbb{Y})$ is strongly anti-coproximinal in $\mathbb{L}(\mathbb{X}, \mathbb{Y}).$ 
 \end{itemize}
\end{cor}

\begin{proof}
	Let us assume that $\mathbb{K}(\mathbb{X}, \mathbb{Y}) \subsetneqq \mathbb{L}(\mathbb{X}, \mathbb{Y}).$ Suppose on the contrary that  $\mathbb{K}(\mathbb{X}, \mathbb{Y})$ is not strongly anti-coproximinal in $\mathbb{L}(\mathbb{X}, \mathbb{Y}).$ Then there exists $S \in \mathbb{L}(\mathbb{X}, \mathbb{Y}) \setminus  \mathbb{K}(\mathbb{X}, \mathbb{Y})$ and $\epsilon \in [0,1)$ such that $ \mathbb{K}(\mathbb{X}, \mathbb{Y}) \perp_B^\epsilon S.$ So, for any $T \in \mathbb{K}(\mathbb{X}, \mathbb{Y})$, $T \perp_B^\epsilon S.$ It is straightforward to see that $T^* \perp_B^\epsilon S^*.$
	Since $\mathbb{Y}$ is reflexive, for any $A \in \mathbb{K}(\mathbb{Y}^*, \mathbb{X}^*)$ there exists $T \in \mathbb{K}(\mathbb{X}, \mathbb{Y})$ such that $T^*= A.$ So, for any $A \in \mathbb{K}(\mathbb{Y}^*, \mathbb{X}^*)$, $A \perp_B^\epsilon S^*.$ This implies $\mathbb{K}(\mathbb{Y}^*, \mathbb{X}^*) \perp_B^\epsilon S^*.$ As $\mathbb{Y}$ is reflexive, applying Theorem \ref{strongly exposed}, either $\mathbb{K}(\mathbb{Y}^*, \mathbb{X}^*)= \mathbb{L}(\mathbb{Y}^*, \mathbb{X}^*)$ or $\mathbb{K}(\mathbb{Y}^*, \mathbb{X}^*)$ is strongly anti-coproximinal in $\mathbb{L}(\mathbb{Y}^*, \mathbb{X}^*),$ both of which lead to a contradiction, thus completing the proof.
\end{proof}

		\begin{prop}\label{Hilbert}
			Let $\mathbb{H}$ be a Hilbert space and let $\mathbb{Y}$ be a proper subspace of $\mathbb{L}(\mathbb{H}).$ Suppose that for any $x,y \in S_{\mathbb{H}},$ there exists an $A \in \mathbb{Y}$ such that $Ax=y$ and $M_A= \{ \pm x\}.$ Then $\mathbb{Y}$ is strongly anti-coproximinal.
		\end{prop}
		
		\begin{proof}
			Suppose on the contrary that $\mathbb{Y}$ is not strongly anti-coproximinal. Then there exists $\epsilon \in [0,1), S \in \mathbb{L}(\mathbb{H}) \setminus \mathbb{Y}$ such that $\mathbb{Y} \perp_B^\epsilon S.$ Let $x \in S_{\mathbb{H}}.$ If $Sx \neq 0,$ let $y= \frac{Sx}{\|Sx\|}.$
			Take $A \in \mathbb{Y}$ such that $Ax=y$ and  $M_{A}=\{x\}.$ Clearly, $\|A\|=1.$ Since $A \perp_B^\epsilon S,$ applying \cite[Th. 3.1]{PSM}, 
			$$
			|\langle Ax, Sx \rangle| \leq  \epsilon \|S\| \|A\| \implies 
			\|Sx\| \leq  \epsilon \|S\|.
			$$
			This is true for every $x \in S_{\mathbb{H}}$ such that $Sx \neq 0.$ Observe that 
			\[
			\|S\|= \sup \{ \|Sx\|: x \in S_{\mathbb{H}}\} \leq \epsilon \|S\| < \|S\|,
			\]
			 a contradiction.
		\end{proof}
		
		\begin{remark}
			Let us consider the proper subspace $\mathbb{Y}$ of $\mathbb{L}(\mathbb{H}),$ consisting of the finite-rank operators. For any $x, y \in S_{\mathbb{H}},$ consider an operator $A \in \mathbb{L}(\mathbb{H})$ such that $Ax=y$ and $Az=0,$ for all $z \in x^\perp,$ where $x^\perp = \{z \in \mathbb{H}: x \perp_B z\}.$ It is trivial to see that $A \in \mathbb{Y}.$ Thus, from Proposition \ref{Hilbert}, we conclude that $\mathbb{Y}$ is a  strongly anti-coproximinal subspace of $\mathbb{L}(\mathbb{H}).$ 
		\end{remark}
		
		We end this article with a characterization of strongly anti-coproximinal subspaces in $\mathbb{L}(\mathbb{X}, \mathbb{Y}),$ whenever $\mathbb{X}, \mathbb{Y}$ are finite-dimensional real polyhedral Banach spaces.

		\begin{prop}
			Let $\mathbb{X}, \mathbb{Y}$ be two finite-dimensional real polyhedral Banach spaces. Then the following are equivalent:
			\begin{itemize}
				\item[(i)] $\mathbb{Z}$ is strongly anti-coproximinal in $\mathbb{L}(\mathbb{X}, \mathbb{Y})$
				\item[(ii)] For any $x \in Ext(B_{\mathbb{X}})$ and $y^* \in Ext(B_{\mathbb{Y}^*}),$ there exists $A \in \mathbb{Z}$ such that $M_{A}=\{ \pm x\}$ and $J(Ax)=\{y^*\}.$
			\end{itemize}
		\end{prop}

		\begin{proof}
			As $\mathbb{X}, \mathbb{Y}$ is finite-dimensional real polyhedral Banach space, $\mathbb{L}(\mathbb{X}, \mathbb{Y})$ is also polyhedral. Moreover, from \cite[Th. 1.3]{RS82} $$Ext(B_{\mathbb{L}(\mathbb{X}, \mathbb{Y})^*})= \{ y^* \otimes x: x \in Ext(B_{\mathbb{X}}), y^* \in Ext(B_{\mathbb{Y}^*})\}.$$
			Using \cite[Th. 2.20]{SSGP3}, $\mathbb{Z}$ is strongly anti-coproximinal in $\mathbb{L}(\mathbb{X}, \mathbb{Y})$ if and only if $\mathcal{J}_{\mathbb{Z}}= Ext(B_{\mathbb{L}(\mathbb{X}, \mathbb{Y})^*}).$ It is now straightforward to check that $y^* \otimes x \in \mathcal{J}_{\mathbb{Z}}$ if and only if there exists $A \in \mathbb{Z}$ such that $M_A=\{ \pm x\}$ and $J(Ax)= \{y^*\}.$ Hence we obtain the desired result.
		\end{proof}
		
			\section*{Declarations}
		
		\begin{itemize}

			\item Conflict of interest
			
			The authors have no relevant financial or non-financial interests to disclose.
			
			\item Data availability 
			
			The manuscript has no associated data.
			
			\item Author contribution
			
			All authors contributed to the study. All authors read and approved the final version of the manuscript.
			
		\end{itemize}
		
		\subsection*{Acknowledgment}
		Shamim Sohel and Souvik Ghosh would like to thank CSIR, Govt. of India, for the financial support in the form of Senior Research Fellowship under the mentorship of Prof. Kallol Paul.

	\end{document}